\documentclass[reqno, 11pt, a4paper]{amsart}

\usepackage{fullpage}
\selectfont

\usepackage[utf8]{inputenc}
\usepackage[OT2,T1]{fontenc}
\usepackage[english]{babel}
\usepackage{amsmath}
\usepackage{amsfonts}
\usepackage{amssymb}
\usepackage{amsthm}
\usepackage{enumerate}

\usepackage[backref=page,pagebackref=true,hyperindex=true,bookmarks=true]{hyperref}

\usepackage[ruled,vlined,boxed]{algorithm2e}

\usepackage[dvipsnames]{xcolor}
\usepackage{graphicx}
\usepackage{tikz}

\usepackage{stmaryrd}

\usepackage[toc,page]{appendix}
\usepackage{graphicx}

\usepackage{xy}
\xyoption{all}

\hypersetup{
  pdfauthor   = {Caruso, Eid, Lercier},
  pdftitle    = {Fast computation of elliptic curve isogenies in
    characteristic two},
  pdfsubject  = {},
  pdfkeywords = {},
  colorlinks=true,
  breaklinks=true, urlcolor=blue, linkcolor=blue, citecolor=blue,
  bookmarksopen=true
}

\newcommand{\N}{\mathbb{N}}
\newcommand{\Z}{\mathbb{Z}}
\newcommand{\Q}{\mathbb{Q}}
\newcommand{\R}{\mathbb{R}}
\newcommand{\ZZ}{\mathbb{Z}_2}
\newcommand{\QQ}{\mathbb{Q}_2}

\newcommand{\FF}{\mathbb{F}_2}
\newcommand{\FFd}{\mathbb{F}_{2^d}}
\newcommand{\ZZd}{\mathbb{Z}_{2^d}}
\newcommand{\GF}[1]{\mathbb{F}_{#1}}

\newcommand{\Kt}{K \llbracket t \rrbracket}
\newcommand{\OKt}{\mathcal{O}_{K} \llbracket t  \rrbracket}
\newcommand{\OLt}{\mathcal{O}_{L} \llbracket t  \rrbracket}

\newcommand{\OK}{\mathcal{O}_{K}}
\newcommand{\bigslant}[2]{{\raisebox{.2em}{$#1$}\left/\raisebox{-.2em}{$#2$}\right.}}

\newcommand{\gf}[1]{\overline{\mathtt{#1}}}
\newcommand{\LL}[1]{L\left[#1\right]}
\newcommand{\vertIII}[1]{{\left\vert\kern-0.25ex\left\vert\kern-0.25ex\left\vert #1
    \right\vert\kern-0.25ex\right\vert\kern-0.25ex\right\vert}}
\DeclareMathOperator{\Char}{char}

\DeclareMathOperator{\Id}{Id}

\DeclareMathOperator{\Inf}{inf}
\DeclareMathOperator{\Hom}{Hom}
\DeclareMathOperator{\End}{End}
\DeclareMathOperator{\Cl}{Cl}
\DeclareMathOperator{\Norm}{Norm}
\DeclareMathOperator{\Tr}{Tr}
\DeclareMathOperator{\Ord}{ord}
\DeclareMathOperator{\Lcm}{lcm}
\def\Ker{\mathop{\rm{Ker}}\nolimits }

\newcommand{\val}{\upsilon_2}
\newcommand{\softO}{\tilde O}
\newcommand{\cA}{\text{\rm A}}
\newcommand{\cM}{\text{\rm M}}
\newcommand{\cC}{\text{\rm C}}
\newcommand{\cT}{\text{\rm T}}

\newcommand{\HU}{$(\mathrm{H}_U)$\xspace}
\newcommand{\HV}{$(\mathrm{H}_V)$\xspace}
\newcommand{\Ep}{$(\mathrm{E}_+)$\xspace}
\newcommand{\Em}{$(\mathrm{E}_-)$\xspace}

\newcommand{\deltag}{\delta g}
\newcommand{\deltaz}{\delta z}

\newcommand{\calV}{\mathcal V}

\newcommand{\RoC}{\text{\rm RoC}}

\newtheorem{defi}{Definition}
\newtheorem{thm}[defi]{Theorem}
\newtheorem*{thm*}{Theorem}
\newtheorem{lem}[defi]{Lemma}
\newtheorem{prop}[defi]{Proposition}
\newtheorem{cor}[defi]{Corollary}

\theoremstyle{remark}
\newtheorem{remark}[defi]{Remark}

\title{Fast computation of elliptic curve isogenies in characteristic two}
\date{\today}

\author[Caruso]{Xavier Caruso}
\address{%
  Xavier Caruso, %
  Univ. Bordeaux, %
  CNRS - IMB - UMR 5251, %
  F-33405 Talence, %
  France. %
  }
\email{xavier.caruso@normalesup.org}

\author[Eid]{\'{E}lie Eid}
\address{%
  \'{E}lie Eid, %
  Univ. Rennes, %
  CNRS, IRMAR - UMR 6625, F-35000
  Rennes, %
  France. %
}
\email{elie.eid@univ-rennes1.fr}

\author[Lercier]{Reynald Lercier}
\address{%
  Reynald Lercier, %
  DGA \& Univ Rennes, %
  CNRS, IRMAR - UMR 6625, F-35000
 Rennes, %
  France. %
}
\email{reynald.lercier@m4x.org}


\thanks{The authors would like to thank the referee for his/her comments
  improving the presentation of the manuscript.}

\begin{document}

\begin{abstract}
  We propose an algorithm that calculates isogenies between elliptic curves
  defined over an extension $K$ of $\QQ$. It consists in
  efficiently solving with a logarithmic loss of $2$-adic precision the first
  order differential equation satisfied by the isogeny.

  We give some applications, especially computing over finite fields of
  characteristic 2 isogenies of elliptic curves and irreducible polynomials,
  both in quasi-linear time in the degree.
\end{abstract}

\maketitle

\section{Introduction}
\label{sec:introduction}

 With the advent of public key cryptography in the 1970s, a keen interest for
 elliptic curves emerged. Pretty soon, attention has been given to their
 isogenies, especially for calculating the number of points of curves defined
 over finite fields~\cite{Schoof95} and more recently for isogeny-based
 cryptography~\cite{RS06,Couveignes06,DFJP14}. Other applications have
 followed: primality proving, normal basis of field extensions, computation of
 irreducible polynomials, finite field isomorphisms,
 \textit{etc.}~\cite{CL09,CEL12,EL13,Narayanan18,BDFDFS19}.
 In this work, we concentrate mainly on algorithms for fields of
 characteristic two that proceed by lifting curves and isogenies to the
 $2$-adics.\smallskip

 Let $k$ be a field and $\ell > 1$ an odd integer. Let $E$ and $\tilde{E}$ be
 two elliptic curves defined over $k$. We suppose that there exists a
 separable isogeny $I : E \to \tilde E$ of degree $\ell$ defined over $k$ as
 well, and we are interested in designing a fast algorithm for computing it.
 When $\Char(k) \neq 2$, this can be achieved by solving a certain nonlinear
 differential equation attached to the
 situation~\cite{bomosasc08,lesi08,lava16}.  Let us recall briefly how it
 works.  It is well known that $E$ and $\tilde{E}$ can be realized by the
 following equations:
 \begin{equation}
   \label{eq:weierstrass}
   E: y^2 = x^3+ a_2\,x^2+ a_4\,x+ a_6
   \quad \text{ and }\quad
   \tilde{E}: y^2 = x^3+ \tilde{a}_2\,x^2+ \tilde{a}_4\,x+ \tilde{a}_6.
 \end{equation}
 These are the so-called Weierstrass models.
 Moreover, by \cite[§~2.4]{kohel96}, we know that the isogeny $I$ has
 an expression of the form
 \begin{equation}
   \label{eq:isogeny}
   I(x,y) = \left( \eta(x),\, c \: y\: \eta'(x) \right)
 \end{equation}
 where $c \in k^\times$ and $\eta$ is a rational function whose numerator and
 denominator have degree $\ell$ and $\ell{-}1$ respectively.  The constant $c$
 is the so-called \emph{isogeny differential} and can be characterized as
 follows: if $I^*:~\Omega_{\tilde{E}/k} \longrightarrow \Omega_{{E}/k}$ is the
 map induced by $I$ on the tangent spaces at $0$, we have
 $\frac{dx} y = c \cdot \frac{d\tilde x}{\tilde y}$. (see~\cite{silv1,silv2}).  The
 terminology ``normalized isogenies'' is sometimes used in the case $c=1$.
 Combining Eqs.~\eqref{eq:weierstrass} and \eqref{eq:isogeny}, we realize that
 the computation of $I$ reduces to solving the following nonlinear
 differential equation:
 \begin{equation}\label{eq:3}
   c^2 \cdot
   (x^3+ a_2\,x^2+ a_4\,x+ a_6)\cdot
   \eta'{}^2 =
   \eta^3 +%
   \tilde{a}_2\,\eta^2 +%
   \tilde{a}_4\,\eta +%
   \tilde{a}_6.
 \end{equation}
 For several reasons, it is convenient to perform the change of variables
 $t = 1/x$ and the change of functions $z(t) = 1/{\eta(x)}$.
 Eq.~(\ref{eq:3}) then becomes
 \begin{equation}
   \label{eq:isog}
   c^2 \cdot (t + a_2 t^2 + a_4 t^3 + a_6 t^4) \cdot z'{}^2 =
   z + \tilde a_2 z^2 + \tilde a_4 z^3 + \tilde a_6 z^4\,.
 \end{equation}
 When $k$ has characteristic $0$, Bostan \textit{et al.}~\cite{bomosasc08}
 proposed to solve Eq.~\eqref{eq:isog} using a well-designed Newton
 iteration. This strategy allows them to compute
 $z(t) \text{ mod } t^{2\ell + 1}$ for a cost of $\softO(\ell)$ operations in
 the ground field and, in a second time, to recover $\eta$ using Pad\'e
 approximants for the same cost.  This approach continues to work well when
 the characteristic of $k$ is positive but large compared to $\ell$. However,
 in the case of small characteristic~$p$, divisions by $p$ do appear and
 prevent the computation to be carried out to its end.  Lercier and
 Sirvent~\cite{lesi08} tackled this issue by lifting $E$, $\tilde E$ and $I$
 to the $p$-adics. In the lifted situation, divisions by $p$ can be performed
 but lead to numerical instability.  One then needs to do a neat analysis of
 the losses of precision.  When $p$ is odd, Lercier and Sirvent showed that
 the number of lost digits stays within $O(\log^2 \ell)$. Later on, still
 assuming that $p$ is odd, Lairez and Vaccon~\cite{lava16} managed to improve
 on this result and came up with a loss of precision in $\log \ell + O(1)$.

 It turns out that extending this approach to characteristic~$2$ is not an
 easy task for a couple of reasons. First of all, the general equation of an
 ordinary elliptic curve in characteristic~$2$ is no longer
 $y^2 = x^3+ a_2 x^2+ a_4 x+ a_6$ but $y^2 + xy = x^3 + a_2 x^2 + a_6$. As a
 consequence, the differential equation we need to study, which is defined
 over the $2$-adics, now takes the form
 \begin{equation}
   \label{eq:isog2}
   c^2 \cdot \Big(4t + (4a_2{+}1) t^2 + 4 a_6 t^4\Big) \cdot z'{}^2
   = 4z + (4 \tilde a_2{+}1) z^2 + 4 \tilde a_6 z^4
 \end{equation}
 for some constants $c, a_2, a_6, \tilde a_2, \tilde a_6$ in the ring of Witt
 vectors of $k$ (if $k$ is the finite field $\FFd$, it is simply $\ZZd$, the
 ring of integers of the unique unramified extension of $\QQ$ of degree $d$).
 Although they look similar, Eq.~\eqref{eq:isog2} is much more difficult to
 handle than Eq.~\eqref{eq:isog}. One reason is structural: the polynomial in
 front of $z'{}^2$, namely $4t + (4a_2{+}1) t^2 + 4 a_6 t^4$, has a root of
 norm $1/4$, meaning that the differential equation we are interested in
 exhibits a singularity in the domain of convergence of the solution we look
 for.  Another reason comes from the exponent $2$ on $z'$, which suggests that
 solving Eq.~\eqref{eq:isog2} will require to extract square roots at some
 point; however, extracting square roots in residual characteristic $2$ is
 known to be a highly unstable operation.  Actually a straightforward
 extension of Lercier and Sirvent's algorithm to characteristic~$2$ leads to
 dramatic losses of $2$-adic precision of order of magnitude $\ell$ (instead
 of $\log \ell$ or $\log^2 \ell$).  The conclusion is that this approach is
 suboptimal and, until now, we were merely reduced to rely on an old algorithm
 by Lercier~\cite{L96} whose theoretical efficiency is subject to heuristics,
 although it behaves surprisingly well in practice (see~\cite{DeFeo11} for a
 discussion on this).

 In this paper, we reconsider the $2$-adic case and propose a new algorithm to
 solve Eq.~\eqref{eq:isog2}. Our algorithm is highly stable and reaches a
 logarithmic loss of $2$-adic precision as Lairez and Vaccon's algorithm
 does. Moreover, it performs very well in practice, allowing for the
 computation of isogenies over $\FF$ of degree up to one million in less than
 one minute. This is the main result of Section~\ref{sec:main-result}.

 \begin{thm}[See Theorem~\ref{thm:stabnonlinear} and
   Proposition~\ref{prop:complexity}]
   \label{thm:main}
 Let $K$ be a finite extension of $\QQ$. Let $\OK$ be its ring of
 integers and $\OK^\times$ the set of invertible elements in $\OK$.
 There exists an algorithm that takes as input
 \begin{itemize}
 \item two positive integers $n$ and $N$,
 \item two elements $a,b \in \OK^\times$\,,
 \item two series $u, v \in \OKt$ with $u(0) \in \OK^\times$
 \end{itemize}
 and, assuming that the differential equation in $z$
 \begin{equation}
 \label{eq:edintro}
 t (t{-}4a) \: u(t)^2 \: z'{}^2 = z\, (z{-}4b) \: v(z)^2
 \end{equation}
 has a unique solution in $t{\cdot}\OKt$, outputs this solution modulo $(2^N, t^n)$
 for a cost of $\softO(n)$ operations in $\OK$ at precision $O(2^M)$ with
 $M = \max(N, 3) + \lfloor \log_2 (n)\rfloor + 2$.
 \end{thm}

\noindent
As a consequence, we obtain efficient algorithms to compute isogenies
between elliptic curves defined over finite fields of characteristic~$2$.
We finally discuss an application of these results to the calculation of
irreducible polynomials defined over such fields in the spirit of the
construction of Couveignes and Lercier~\cite{CL13}.

\medskip

This article is supplemented by an appendix of theoretical flavor, in which
we reuse the techniques of $p$-adic precision introduced in the core of
the paper to prove that the radius of convergence of the solution of
Eq.~\eqref{eq:edintro} varies continuously with $u$. To some extent, this
result can be understood as the theoretical essence at the origin of the
excellent behaviour of our main algorithm.
Indeed, the assumption on $z$ made in our main theorem roughly
means that $z$ has a radius of convergence much larger than expected;
the fact that this radius of convergence remains large when the input
is perturbed is the key property behind the numerical stability
of the algorithm.

\section{Fast resolution of a $2$-adic differential equation}
\label{sec:main-result}

This section is devoted to the effective resolution of the nonlinear
differential equation \eqref{eq:edintro}, leading eventually to the proof of
our main theorem.  In more details, the computation model we will use
throughout this paper is introduced in Section~\ref{subsec:compmodel}. The two
next subsections are concerned with preliminary material: we show that
Eq.~\eqref{eq:edintro} has a unique solution in certain cases and study a pair
of linear differential equations that will eventually play a quite important
role.  Our algorithm is presented in Section~\ref{subsec:algo} and the proof
of its correctness is exposed in
Section~\ref{subsec:precision-analysis}. Finally the implementation and
corresponding timings are discussed in Section~\ref{sec:experiments}.

Throughout this section, the letter $K$ refers to a fixed algebraic
extension of $\QQ$. We recall that the $2$-adic valuation extends
uniquely to $K$; we will denote it by $\val$ and will always assume
that it is normalized by $\val(2) = 1$.
We let $\OK$ denote the ring of integers of $K$ and $\pi \in \OK$
be a fixed uniformizer of $K$. We reserve the letter $e$ for
the ramification index of the extension $K/\QQ$, so that we have
$\val(\pi) = 1/e$.
It will be convenient to extend the valuation to quotients of
$\OK$: if $x \in \OK/\pi^{eM}\OK$, we define $\val(x) = M $
when $x = 0$ and $\val(x) = \val(\hat x)$ where $\hat x \in \OK$ is
any lifting of $x$ otherwise.

\subsection{Computation model}
\label{subsec:compmodel}

Carrying explicit computations in $K$ is not straightforward because
elements of $K$ carry an infinite amount of information and need to
be truncated to fit in the memory of a computer: we sometimes say
that $K$ is an \emph{inexact} field.
Over the years, several computation models have been proposed to
handle these difficulties: interval arithmetic, floating point
arithmetic, lazy arithmetic, \emph{etc.}
We refer to \cite{caruso17} for a detailed discussion about this,
including many examples illustrating the advantages and the
disadvantages of each possible model.

Throughout this article, we will use the \emph{fixed point arithmetic}
model at precision $O(2^M)$, where $M$ is a fixed positive number in
$\frac 1 e\, \Z$.
Concretely, this means that we shall represent elements of $K$ by
expressions of the form $x + O(2^M)$ with $x \in \OK/\pi^{eM}\OK$.
Additions, subtractions and multiplications are defined
straightforwardly:
\begin{align*}
\big( x + O(2^M) \big) + \big( y + O(2^M) \big) & = (x + y) + O(2^M)\,, \\
\big( x + O(2^M) \big) - \big( y + O(2^M) \big) & = (x - y) + O(2^M)\,, \\
\big( x + O(2^M) \big) \times \big( y + O(2^M) \big) & = xy + O(2^M)\,.
\end{align*}
The specifications of division go as follows: for $x, y \in
\OK/\pi^{eM}\OK$, the division of $x + O(2^M)$ by $y + O(2^M)$
\begin{itemize}
\item raises an error if $\val(y) > \val(x)$,
\item returns $0 + O(2^M)$ if $x = 0$ in $\OK/\pi^{eM}\OK$,
\item returns any representative $z + O(2^M)$ with the property
$x = yz$ in $\OK/\pi^{eM}\OK$ otherwise.
\end{itemize}

\subsubsection*{Complexity notations and assumptions.}

In what follows, we shall always assume that we can perform additions,
subtractions, multiplications and divisions in the computation model described
above. Let $\cA(K; M)$ be an upper bound on the bit complexity of algorithms
that carry out these arithmetic operations.  When $K = \QQ$, the quotients
$\OK / \pi^{eM}\OK$ are just $\Z/2^M\Z$ and, relying on fast Fourier
transform, we can take $\cA(\QQ; M) \in \softO(M)$ (where the
$\softO$-notation means that we are hiding logarithmic factors).  More
generally, if $K$ is an extension of $\QQ$ of degree $d$ which is presented
either by a polynomial which remains irreducible modulo $2$ (unramified case)
or by an Eisenstein polynomial (totally ramified case), elements of
$\OK/2^{eM}\OK$ can be represented safely as polynomials over $\Z/2^M\Z$ of
degree at most $d$ and we can take $\cA(K; M) \in \softO(dM)$ (reducing
polynomial multiplications to integer ones with Kronecker
substitution method~\cite{kronecker82,schonhage82}). Finally, the same
estimates remain valid when $K$ is presented as a two-step extension, the
first one being given by an ``unramified'' polynomial and the second one being
given by an Eisenstein polynomial. We note that this covers all extensions of
$\QQ$.

We further assume that we are given a division-free algorithm for
multiplying polynomials over any exact base ring and we let $\cM(n)$ be a
bound on its algebraic complexity (\emph{i.e.} the number of arithmetical
operations in the base ring it performs). For convenience, we will also
suppose that the function $\cM$ satisfies the superadditivity
assumption, that is:
$$\forall n, n' \in \N, \quad \cM(n + n') \geq \cM(n) + \cM(n').$$
Standard algorithms allow us to take $\cM(n) \in \softO(n)$.
Besides, we observe that an algorithm as above can be used to multiply
polynomials over $K$ in the fixed point arithmetic model since additions,
multiplications and divisions in this model all reduce to the similar
operations in the \emph{exact} quotient ring $\OK/\pi^{eM}\OK$.
As a consequence, when working in the fixed point arithmetic model at
precision $O(2^M)$, the bit complexity of the multiplication of two
polynomials of degree $n$ over $K$ is bounded by above by
$\cM(n) \cdot \cA(K;M)$, which itself stays within $\softO\big(nM
\cdot
[K\:{:}\:\QQ]\big)$ under standard assumptions.

\subsection{The setup}
\label{subsec:nonlinear}

Let $\Kt$ be the ring of formal series over $K$ (in the variable
$t$). Given two series $U,V \in \Kt$, we consider the following
nonlinear differential equation whose unknown is~$z$:
\begin{equation}
\label{eq:nonlinear}
U \cdot z'{}^2 = V \circ z\,.
\end{equation}
When $V$ is an actual series, the composite $V \circ z$ is not
always well defined; however, it is as soon as $z$ vanishes at $0$,
\emph{i.e.} $z \in t \Kt$. For this reason, in what follows, we will
always look for solutions of~(\ref{eq:nonlinear}) in $t \Kt$. We will also always
assume that \emph{both $U$ and $V$ have $t$-adic valuation $1$}; in
other words, we suppose that there exist \emph{nonzero} scalars $u_1,
v_1 \in K$ such that $U = u_1 t + O(t^2)$ and $V = v_1 t + O(t^2)$.

The following proposition shows that these assumptions are enough
to guarantee the existence and the uniqueness of a solution
to Eq.~\eqref{eq:nonlinear}.

\begin{prop}
  \label{prop:solnonlinear}
  Assuming that $U$ and $V$ have $t$-adic valuation $1$, the
  differential equation \eqref{eq:nonlinear} admits a unique nonzero
  solution in $t\Kt$.
\end{prop}

\begin{proof}
  Write $U = \sum\limits_{n=1}^{\infty} { u_n\, t^n}$ and
  $V = \sum\limits_{n=1}^{\infty} { v_n\, t^n}$. We are looking
  for a solution of Eq.~\eqref{eq:nonlinear} of the form
  $z = \sum \limits _{n=1}^{\infty} { z_n\, t^n}$.
  Taking the $n$-th derivative of Eq.~\eqref{eq:nonlinear} (and using
  Fa\`a di Bruno's formula to evaluate the successive derivatives
  of $V \circ z$), we end up with the relations
  \begin{multline}
  \label{eq:faadibruno}
    \sum \limits_{i=0}^{n-1} \sum \limits _{j=0}^i
    (j+1)(i-j+1)z_{j+1}z_{i-j+1}u_{n-i} =\\
    \sum \limits _{k_1 + 2k_2 + \cdots +
      nk_n=n} \frac{(k_1+k_2+\cdots + k_n)!}{k_1!\, k_2!\, \cdots \,k_n!}\, v_{k_1 +
      k_2 + \ldots +k_n}\, z_1 ^{k_1} z_2^{k_2} \cdots z_n ^{k_n}\,.
  \end{multline}
  When $n = 1$, this formula reduces to $u_1 z_1^2 = v_1 z_1$, showing that
  $z_1$ must be equal to $0$ or ${v_1}/{u_1}$ since $u_1$ and $v_1$ are
  nonzero scalars. For bigger $n$, we observe that the coefficient $z_n$ only
  appears in the summands indexed by $(i,j) = (n{-}1, 0)$ and
  $(i,j) = (n{-}1, n{-}1)$ in the left hand side of Eq.~\eqref{eq:faadibruno}
  and the summand indexed by $(k_1, \ldots, k_n) = (0, \ldots, 0, 1)$ in the
  right hand side.  Isolating them since $v_1\ne 0$ by hypothesis, one obtains
  $$z_n = \frac{P_n(z_1,\ldots,z_{n-1})}{(2n{-}1)\:v_1}$$
  for some polynomial $P_n \in K[X_1, \ldots, X_{n-1}]$ vanishing at
  $(0, \ldots, 0)$. It follows for this observation that $z$ must
  vanish if $z_1$ vanishes. Otherwise, the coefficients $z_n$ are
  all uniquely determined, then showing the existence and the
  unicity of a nonzero solution to Eq.~\eqref{eq:nonlinear}.
\end{proof}

We now introduce the two following additional assumptions:
\begin{itemize}
\item[\HU:]
there exists $a \in \OK^\times$ and $u \in \OKt$ with $u(0) \in \OK^\times$
s.t. $U(t) = t(t{-}4a) \cdot u(t)^2$;
\item[\HV:]
there exists $b \in \OK^\times$ and $v \in \OKt$ with $v(0) \in \OK^\times$
s.t. $V(t) = t(t{-}4b) \cdot v(t)^2$.
\end{itemize}

\begin{remark}
\label{rem:HUHV}
  Both Assumptions \HU and \HV are fulfilled for the differential
  equation~\eqref{eq:isog2} (which is the one that we need to solve in
  order to compute isogenies between elliptic curves) after possibly
  replacing $K$ by its unramified extension of degree $2$.
  Indeed $U$ is then
  the polynomial $c^2 \big(4t + (4 a_2 + 1) t^2 + 4 a_6 t^4\big)$ for
  some $c \in \OK^\times$ and some $a_2, a_6 \in \OK$. Looking at valuations,
  we find that its Newton polygon has a segment of slope $-2$. Consequently
  $U$ has a root of valuation $2$, \emph{i.e.} $U$ is divisible by $(t{-}4a)$
  for some $a \in \OK^\times$. Since $U$ is also obviously divisible by $t$,
  we find that
$U(t) = c^2 \cdot t(t{-}4a) \cdot U_0(t)$
where $U_0$ is the polynomial of degree $2$ explicitly given by
$$U_0(t) = 4 a_6 \,t^2 + 16 a a_6 \,t + (64 a^2 a_6 + 4 a_2 + 1).$$
In particular, we observe that $U_0(t) \equiv 1 \pmod 4$ and $U_0(0)
\equiv 1{+}4 a_2 \pmod 8$.
This ensures that $U_0$ admits a square root in $\OLt$ with
$L = K[\sqrt{1{+}4a_2}]$. It is easy to check that $L = K[\mu]$
where $\mu$ is a root of the polynomial $P(X) = X^2{-}X{-}a_2$.
Since $P$ is separable modulo $2$, we deduce that $L$ is unramified
over $K$. More precisely, if $\text{Tr}_{K/\Q_2} (a_2)$ is odd, $L$
is the unique unramified extension of $K$ of degree $2$ and $L=K$
otherwise.
Letting finally $u = c \: \sqrt{U_0}$, we find that \HU is satisfied
over $\OLt$.
The fact that \HV is satisfied as well is proved similarly.
\end{remark}

\begin{remark}
\label{rem:HUHVbis}
We may further remark that an ordinary elliptic curve
$E/\FFd: y^2 + xy = x^3 + a_2 x^2 + a_6$ is the twist of the elliptic curve
$E'/\FFd: y^2 + xy = x^3 + a_6$ up to the twisting isomorphism
$(x,y)\mapsto (x,y{+}sx)$ where $s$ is a solution of the equation
$s^2+s+a_2 = 0$ (possibly defined over a quadratic extension of $\FFd$).
Since \HU and \HV are fulfilled over $K$ for the $2$-adic differential
equation obtained from $E'$ and the twist $\tilde{E}'$ of the isogenous curve
$\tilde{E}$, we can avoid the transition by the quadratic extension $L$ for
solving Eq.~(\ref{eq:isog2}) between $E'$ and $\tilde{E}'$, even if a
quadratic extension may be finally needed to obtain the isogeny between $E$
and $\tilde{E}$ by applying the twisting isomorphisms.
\end{remark}

In the next subsections, we are going to design an efficient algorithm
to compute the unique solution of Eq.~\eqref{eq:nonlinear} under Assumptions
\HU and \HV.

\subsection{Two linear differential equations}
\label{subsec:linear}

We introduce two auxiliary linear differential equations that will appear
later on as important ingredients in the resolution of the nonlinear
differential equation~\eqref{eq:nonlinear}. Precisely, given
$a \in \OK^\times$, we consider
$$\begin{array}{r@{:\hspace{2em}}l}
\text{\Ep} & t (t{-}4a)\, y' +  (t{-}2a)\, y = f\,, \smallskip \\
\text{\Em} & t (t{-}4a)\, y' -  (t{-}2a)\, y = f\,,
\end{array}$$
where $y$ is the unknown and the right hand side $f$ lies in $\Kt$.

In the following, if $n$ is a nonnegative integer, we denote by $S_2(n)$ the
sum of its digits in base~$2$. For example $S_2(3) = 2$. One easily checks that
the inequality $S_2(n) \leq \lfloor \log_2(n{+}1) \rfloor $ is valid for all
$n \geq 0$ (here $\lfloor x \rfloor$ denotes the integer part of $x$) and that the
equality holds if and only if $n = 2^m -1$ for some $m$.

\begin{prop}
\label{prop:linear}
For any $f = \sum\limits_{i=0}^\infty f_i t^i \in \Kt$, the differential
equation \text{\Ep} (resp. \Em) admits a unique solution in~$\Kt$. Moreover
\begin{itemize}
\item
if $y = \sum\limits_{i=0}^\infty y_i t^i$ is the solution of \Ep,
we have
$$\begin{array}{r@{\quad}l}
\forall i \geq 0, &
\val(y_i) \geq \min \limits _{0 \leq k \leq i}
\val(f_k) - \lfloor \log_2(i{+}1)\rfloor - 1\,;
\end{array}$$
\item
if $y = \sum\limits_{i=0}^\infty y_i t^i$ is the solution of \Em,
we have
$$\begin{array}{r@{\quad}l}
&
\val(y_0) = \val(f_0) - 1\,, \smallskip \\
&
\val(y_1) \geq \min\big(\val(f_0) - 2, \, \val(f_1) - 1\big)\,, \smallskip \\
\forall i \geq 2, &
\val(y_i) \geq \min \limits _{2 \leq k \leq i}
\val(f_k) - \lfloor\log_2(i{-}1)\rfloor - 1\,.
\end{array}$$
\end{itemize}
\end{prop}

\begin{proof}
We only treat the equation \Ep, the case of \Em being totally similar.
Plugging $f = \sum\limits_{i=0}^\infty f_i t^i$ and
$y = \sum\limits_{i=0}^\infty y_i t^i$ in \Ep,
we obtain
\begin{equation}
  y_0 = -\frac{f_0}{2a}, \quad
  y_i = \frac{i\, y_{i-1} - f_i} {2a \cdot (2i{+}1)}\quad \text{for }i>0.
\end{equation}
The existence and the unicity of the solution of \Ep follows.
Regarding the growth of the coefficients, an easy
induction on $i$ shows that $y_i$ can be written in the form
$$y_i= \sum \limits _{k=0}^{i} \frac{2^{k-i-1}\cdot i!}{k!}
\, e_{i,k}\, f_k$$ with $e_{i,k} \in \OK$ for all $i$ and $k$.  We conclude by
applying Euler's formula,

\medskip

\noindent\hfill
  $\displaystyle
   \upsilon _2 \left( \frac{2^{k-i-1}i!}{ k!} \right) =
    (k-i -1)+ (i- S_2(i)) - (k- S_2(k))
  = -S_2(i) +S_2(k) -1\,.$
\hfill
\end{proof}

\medskip

The first part of Proposition~\ref{prop:linear} allows us to define the
function $\psi_+ : \Kt \to \Kt$ (resp. $\psi_- : \Kt \to \Kt$) taking
$f$ to the unique solution of the differential equation \Ep (resp. \Em).
Clearly $\psi_+$ and $\psi_-$ are $K$-linear mappings.
Moreover, given a positive integer $n$, Proposition~\ref{prop:linear}
again shows that $\psi_+$ and $\psi_-$ map $t^n \Kt$ to itself and then
induce $K$-linear endomorphisms $\psi_{+,n}$ and $\psi_{-,n}$ of
$\Kt/(t^n)$.

\begin{lem}
\label{lem:psipm}
For all $f \in \Kt$, we have the relation
$$t(t{-}4a) \cdot \psi_+(f) = \psi_-\big(t(t{-}4a) \cdot f\big).$$
\end{lem}

\begin{proof}
It is enough to check that $t(t{-}4a) \: \psi_+(f)$ is a solution of
$t (t{-}4a)\, y' -  (t{-}2a)\, y = t(t{-}4a) \: f$,
which is a direct computation.
\end{proof}

\begin{figure}[htbp]
  \begin{center}
    \parbox{0.85\linewidth}{%
      \begin{footnotesize}\SetAlFnt{\small\sf}%
        \begin{algorithm}[H]%
          \caption{Linearized equation solver} %
          \label{algo:LinSolve}%
          \SetKwInOut{Input}{Input} %
          \SetKwInOut{Output}{Output} %
          \SetKwProg{LinDiffSolve}{\tt LinDiffSolve}{}{}%
          \LinDiffSolve{$(a,f,n)$}{
            \Input{$a \in \OK^\times$, $n \in \N$ and $f = \sum\limits_{i=0}^{n-1} f_i t^i \in \Kt/(t^n)$}
            \Output{$\psi_{+,n}(f)$} \BlankLine %
	$y_0 := \frac{-f_0}{2a}$

            \For{$i := 1\ \mathrm{\mathbf{to}}\ n-1$}
            {
              $y_i := \displaystyle \frac{i y_{i-1} - f_i}{ 2a \cdot (2i{+}1)}$\,\;
            }
            \KwRet{$\sum \limits _{i=0}^{n-1} {y_i t^i}$\;}
          }
        \end{algorithm}
      \end{footnotesize}
    }
  \end{center}
\end{figure}

We now move to the effective computation of $\psi_+$. Following the proof of
Proposition~\ref{prop:linear}, we directly get Algorithm~\ref{algo:LinSolve}
(\texttt{LinDiffSolve}), whose numerical stability is studied in
Proposition~\ref{prop:stablinear} hereafter.  Before stating it, let us recall
that $e$ denotes the ramification index of $K$ over $\QQ$ and that, given
$N \in \frac 1 e \, \N$, we use the notation $O(2^N)$ to refer to a quantity
which is divisible by $\pi^{eN}$.

\begin{prop}
\label{prop:stablinear}
Let $n \in \N$, $N \in \frac 1 e \, \N$ and $f \in \OKt/(t^n)$.
We assume that $\psi_{+,n}(f) \in \OKt/(t^n)$. Then,
when $\textrm{\tt LinDiffSolve}(a,f,n)$ is run with fixed point arithmetic
at precision $O(2^M)$ with $M = N + \lfloor \log _2(n{+}1) \rfloor + 1$,
all the performed computations are done in $\OK$ and the result is
correct at precision $O(2^N)$.
\end{prop}

\begin{proof}
The fact that all the computations stay within $\OK$ is a direct
consequence of the assumption that $\psi_{+,n}(f)$ has coefficients
in $\OK$.
Let $y$ be the output of $\textrm{\tt LinDiffSolve}(a,f,n)$.
It follows from the definition of fixed point arithmetic that $y$ is
solution of
$$t (t{-}4a) y' +  (t{-}2a) y = f + h$$
for some $h \in \pi^{eM} \OKt/(t^n)$. Consequently $y = \psi_{+,n}(f+h) =
\psi_{+,n}(f) + \psi_{+,n}(h)$.
On the other hand, Proposition~\ref{prop:linear}
(applied with $h$) shows that $\psi_{+,n}(h) \in \pi^{eN} \OKt/(t^n)$. Hence
$y \equiv \psi_{+,n}(f) \pmod{(\pi^{eN}, t^n)}$, which exactly means
that $y$ is correct at precision $O(2^N)$.
\end{proof}

\begin{remark}
\label{rem:lindiffsolveval}
It follows from the specifications on our computation model that, when
the $m$ first coefficients of $f$ vanish, the $m$ first coefficients of
the output of $\textrm{\tt LinDiffSolve}(a,f,n)$ vanish as well.
\end{remark}

\subsection{The algorithm}
\label{subsec:algo}

We go back to the nonlinear differential equation \eqref{eq:nonlinear} and
assume the hypothesis \HU.  In this setting, we will construct the solution by
successive approximations using a Newton scheme.  In order to proceed, we
suppose that we are given $z_m \in \Kt$ for which Eq.~\eqref{eq:nonlinear} is
satisfied modulo $t^m$. We look for a more accurate solution $z_n$ of the form
$z_n = z_m + h$ with $h \in t^m \Kt$. We compute
\begin{align*}
U(t) \cdot z'_n{}^2 = U(t)\cdot (z'_m + h')^2
 & \equiv U(t)\cdot (z'_m{}^{\!2} + 2\, z'_m h') \pmod{t^{2m-1}}\,, \\
V(z_n) = V(z_m + h)
 & \equiv V(z_m) + V'(z_m) \cdot h  \pmod{t^{2m-1}}\,.
\end{align*}
Identifying both terms, we obtain the relation
\begin{equation}
\label{eq:h}
2 \: U(t)\: z'_m \cdot h' - V'(z_m) \cdot h
\:\equiv\: V(z_m) -U(t)\cdot z'_m{}^{\!2}  \pmod{t^{2m-1}}\,.
\end{equation}
By assumption, we know that $U(t) \cdot z'_m{}^{\!2} \equiv V(z_m) \pmod
{t^m}$. Differentiating this equation and dividing by $z'_m$, we
obtain $V'(z_m) \equiv U'(t) \:z'_m + 2 \:U(t)\: z''_m \pmod{t^{m-1}}$.
Plugging this congruence in Eq.~\eqref{eq:h}, we find
$$
2 \: U(t)\: z'_m \cdot h' - \Big( U'(t) \:z'_m + 2 \:U(t)\: z''_m \Big) \cdot h
\:\equiv\: V(z_m) -U(t)\cdot z'_m{}^{\!2}  \pmod{t^{2m-1}}\,.
$$
Replacing $U(t)$ by $t(t{-}4a)\:u(t)^2$ thanks to Hypothesis \HU, and setting
$h = z'_m u\cdot y$, we end up with the differential equation in $y$
\begin{displaymath}
t(t{-}4a)\: y' - (t{-}2a)\: y
\:\equiv\: f_n \pmod{t^{2m-1}} \quad \text{with} \quad f_n = \frac{1}{2
  u(t)^3} \left(\frac{V(z_m)}{z'_m{}^{\!2}} - U(t) \right).
\end{displaymath}
By the results of Section~\ref{subsec:linear}, we derive
$h \equiv z'_m u \cdot \psi_-(f_n) \pmod{t^{2m-1}}$. Repeating the above
calculations in the reverse direction, we obtain the next proposition.

\begin{prop}
  \label{prop:Newton}
  We assume \HU.
  Let $m>1$ be an integer and let $z_m \in \Kt$ be a solution of
  Eq.~\eqref{eq:nonlinear} modulo $t^m$. Then
\begin{equation}
\label{eq:oldnewton}
z_m + z'_m u \cdot
  \psi_-\left(\frac{1}{2 u(t)^3}
  \left(\frac{V(z_m)}{z'_m{}^{\!2}} - U(t) \right)\right)
\end{equation}
  is a solution of Eq.~\eqref{eq:nonlinear} modulo $t^{2m-1}$.
\end{prop}
It would be reasonable to expect that Proposition~\ref{prop:Newton} could be
easily turned into an algorithm that solves the nonlinear differential
equation \eqref{eq:nonlinear}. However, for several reasons (related to
the precision analysis), we shall modify a bit our Newton iteration.
From now on, we assume \HV, set $\lambda = b a^{-1} \in \OK^\times$.
For $z_m \in t\Kt/(t^m)$, we write
\begin{equation}
  \label{eq:ChangeOfIteration}
  z_m = \lambda t + t(t{-}4a)\, q_m
\end{equation}
with $q_m \in \Kt/(t^{m-1})$.
So, $z_m$ is a solution of Eq.~\eqref{eq:nonlinear}
modulo $t^m$ if and only if $q_m$ satisfies
\begin{equation}
  \label{eq:nonlineartilde}
  W(t, q_m) \equiv u(t)^2 z'_m{}^{\!2} \pmod {t^{m-1}}\,.
\end{equation}
where $W$ is defined by
\begin{displaymath}
  W(t,x)= \big(\lambda + (t{-}4a)x \big ) \cdot
          \big ( \lambda +tx \big) \cdot
          v^2 \big ( \lambda t + t(t{-}4a)x \big)
  \in \OK \llbracket t, x \rrbracket\,.
\end{displaymath}
Rewriting Proposition~\ref{prop:Newton}, we obtain the following corollary.
\begin{cor}
\label{cor:Newton}
We assume \HU and \HV.  Let $m>1$ be an integer and let $q_m$ be a solution
of Eq.~\eqref{eq:nonlineartilde} modulo $t^{m}$. Then
\begin{displaymath}
  q_m +  z'_m u \cdot
  \psi_+\left(\frac{1}{2 u(t)^3}
    \left(\frac{W(q_m)}{z'_m{}^{\!2}} - u(t)^2 \right)\right)
\end{displaymath}
is a solution of Eq.~\eqref{eq:nonlineartilde} modulo $t^{2m-1}$.
\end{cor}
\begin{proof}
  The formula is easily obtained by plugging Eq.~\eqref{eq:ChangeOfIteration}
  in Eq.~\eqref{eq:oldnewton}, and using Lemma~\ref{lem:psipm}.
\end{proof}

With Corollary~\ref{cor:Newton} and a small optimization consisting in
integrating the computation of $(z'_m)^{-2}$ in our Newton scheme,
we get Algorithm~\ref{algo:DiffSolve} (\texttt{DiffSolve}) and
Algorithm~\ref{algo:IsoSolve} (\texttt{IsoSolve}).
\begin{figure}[htbp]
  \begin{center}
    \parbox{0.85\linewidth}{%
      \begin{footnotesize}\SetAlFnt{\small\sf}%
        \begin{algorithm}[H]%
          \caption{Non linear differential equation solver} %
          \label{algo:DiffSolve}%
          \SetKwInOut{Input}{Input} %
          \SetKwInOut{Output}{Output} %
          \SetKwProg{DiffSolve}{\texttt{DiffSolve}}{}{}%
          \DiffSolve{$(a,\, u,\, u^2,\, u^{-3},\, b,\, v^2,\, n)$}{
            \Input{$u, u^2, v^2 \bmod t^n,\ u, u^{-3} \bmod t^{\lceil n/2
                \rceil},\ a$ and $b$ satisfying \HU and \HV}
            \Output{$q_n$ ${}\bmod t^n$, $(z'_n)^{-2} \bmod t^{\lceil n/2 \rceil}$}
            \BlankLine %
            \BlankLine %
            \If{ $n \leq 1$}%
            {%
              \KwRet{$v_1u_1^{-1} \big ( b a^{-1} - 4a \big )^{-1} \bmod t^n,\, 0 \bmod t^{n-1}$}\;%
            }%
            {}
            $ m := \lceil \frac{n{-}1}2 \rceil $\;
            $q_m,\, r_m := \texttt{DiffSolve}(a,\, u,\, u^2,\, u^{-3},\, b,\, v^2,\, m)$%
            \tcp*{recursive call}
            $z_m:= ba^{-1}t + t(t-4a)q_m \bmod t^{n+1}$\;

            $ w_n := z'_m{}^{\!2} \bmod t^{n}$				%
            \tcp*{1 coeff.\,lost}
            $ s_n := u^2\cdot w_n - W\circ q_m\ \bmod t^{n}$%
            \tcp*{$s_n \bmod t^{m} = 0$}
            $r_n :=  r_{m} \cdot (2 - r_{m}\cdot w_n) \bmod t^{n-m}$%
            \tcp*{$(z_n')^{-2}$ (Newton iter.)}
            $f_n := 2^{-1} \,s_n \cdot r_n \cdot u^{-3} \bmod  t^n$%
            \tcp*{The argument of $\psi_+$}

            $y_n := \texttt{LinDiffSolve}(a,\, f_n, \, n)$%
            \tcp*{$y_n \bmod t^{m} = 0$}
            \KwRet{$q_m + z'_m \cdot u \cdot y_n \bmod t^{n}$,\ $r_n$}%
          }
       \end{algorithm}
     \end{footnotesize}
   }
 \end{center}
\end{figure}

\begin{figure}[htbp]
  \begin{center}
    \parbox{0.85\linewidth}{%
      \begin{footnotesize}\SetAlFnt{\small\sf}%
        \begin{algorithm}[H]%
          \caption{Isogeny differential equation solver} %
          \label{algo:IsoSolve}%
          \SetKwInOut{Input}{Input} %
          \SetKwInOut{Output}{Output} %
            \BlankLine %
          \SetKwProg{IsoSolve}{\texttt{IsoSolve}}{}{}%

          \IsoSolve{$(U,\,V,\,n)$}{
            \Input{$U, V \bmod t^n$ satisfying Assumptions \HU and \HV}
            \Output{the solution ${z_n}$ ${}\bmod t^n$ of Eq.~\eqref{eq:nonlinear}}
            \BlankLine %
            \BlankLine %
            Compute $a$, $b$, $U_0 \bmod t^{n-1}$ and $V_0 \bmod t^{n-1}$
            \tcp*{$U_0=u^2$ and $V_0=v^2$}

            Compute $u \bmod t^{\lceil (n-1)/2 \rceil}$ and $u^{-3} \bmod t^{\lceil (n-1)/2 \rceil}$\;
            $z := \texttt{DiffSolve}(a,\, u,\, u^2,\, u^{-3},\, b,\, v^2,\, n-1)$\;
            \KwRet{ $ba^{-1}t + t(t-4a)\, z \bmod t^{n}$ }%
          }
       \end{algorithm}
     \end{footnotesize}
   }
 \end{center}
\end{figure}
\smallskip

If we could work at infinite $p$-adic precision, it would be clear that
Algorithm~\ref{algo:IsoSolve} is correct. The next theorem shows that its
correction still holds in the fixed point arithmetic model.

\begin{thm}
\label{thm:stabnonlinear}
Let $n \in \N$, $N \in \frac 1 e \, \N$ and $U, V \in \Kt$. We assume \HU and \HV
and that the unique nonzero solution of Eq.~\eqref{eq:nonlinear} has
coefficients in $\OK$. Then,
when $\textrm{\tt IsoSolve}(U,V,n)$ runs with fixed point arithmetic
at precision $O(2^M)$ with $M = \max(N,3) +
\lfloor \log _2(n) \rfloor + 2$, all the performed computations
are done in $\OK$ and the result is correct at precision $O(2^N)$.
\end{thm}

We delay the proof of Theorem~\ref{thm:stabnonlinear} to
Section~\ref{subsec:precision-analysis}. Let us first study the complexity of
the Algorithms~\ref{algo:DiffSolve} and~\ref{algo:IsoSolve}.
We recall that $\cM(n)$ denotes the algebraic complexity
of a feasible algorithm that computes the product of two polynomials on degree
$n$.  Similarly, given a fixed series $W \in \Kt$, we define $\cC_W(n)$ as the
algebraic complexity of an algorithm computing the composite $W \circ z$
modulo $t^n$. In our case of interest, $W$ turns out to be a polynomial of
degree $4$ and $\cC_W(n) = O(\cM(n))$. More generally, we observe that, when
$W$ is a polynomial of degree $d$, we have $\cC_W(n) = O(d \cM(n))$. In what
follows, we assume that $\cC_W$ satisfies the superadditivity hypothesis,
\emph{i.e.} that $\cC_W(n+n') \geq \cC_W(n) + \cC_W(n')$ for all integers $n$
and $n'$.

\begin{prop}
\label{prop:complexity}
When it is called on the input $(U, V, n)$, the algorithm
\textrm{\tt IsoSolve} performs at most $O\big(\cM(n) + \cC_W(n)\big)$
operations in $K$.
\end{prop}

\begin{proof}
  The calculation of $a$ and $b$ is done with the Hensel
  lifting algorithm applied to $U(t)$.  The series $u(t)^2$ and $v(t)^2$ are then obtained
  by Euclidean division by $t(t{-}4a)$. Then, the series $u(t)^{-1}$ can be
  computed by the Newton iteration
  $r \mapsto r \cdot (3 - u^2 r^2)/2$. Finally $u(t)$ is
  obtained by multiplying $u(t)^{-1}$ by $u(t)^2$ and $u(t)^{-3}$ is
  obtained by cubing $u(t)^{-1}$. The total complexity of the
  precomputation steps is then at most $O(\cM(n))$ operations in $K$.

  Examining the core of Algorithm~\ref{algo:DiffSolve}, we find that
  its algebraic complexity $\cT(n)$ satisfies the relation
  $$\cT(n) \leq \cT\left(\left\lceil\frac{n+1} 2\right\rceil\right) +
    O\big(\cM(n) + \cC_W(n)\big)\,,$$
  the complexity of \texttt{LinDiffSolve} being linear in $n$ and hence
  dominated by $\cM(n)$. Solving the recurrence and using the
  superadditivity of $\cM$ and $\cC_W$, we find
  $\cT(n) = O\big(\cM(n) + \cC_W(n)\big)$ which is also the complexity of Algorithm~\ref{algo:IsoSolve}.
\end{proof}

\begin{cor}
When executed with fixed point arithmetic at precision $O(2^M)$,
the bit complexity of the algorithm \textrm{\tt IsoSolve} is
$O\big((\cM(n) + \cC_W(n))\cdot \cA(K; M)\big)$.
\end{cor}

\begin{proof}
It is a direct consequence of Proposition~\ref{prop:complexity}
(combined with the fact that the three algorithms only
performs operations in $\OK$ as promised by
Theorem~\ref{thm:stabnonlinear}).
\end{proof}

\begin{remark}\label{rmk:1}
  Since \texttt{IsoSolve} is of quasi-linear complexity in $n$, it is faster
  to compute the composition of two isogenies with
  Algorithm~\ref{algo:IsoSolve} than to compute each isogeny independently and
  then perform their compositions.
\end{remark}

\subsection{Precision analysis}
\label{subsec:precision-analysis}

The aim of this subsection is to prove Theorem~\ref{thm:stabnonlinear}.
The general scheme of our proof follows that of Lairez and Vaccon
\cite{lava16} and relies
mostly on the theory of ``differential precision'' developed by
Caruso, Roe and Vaccon in~\cite{carova14,carova15}.
Recall that the differential equation we have to solve reads
$$t(t{-}4a) \cdot u(t)^2 \cdot z'{}^2 = V(z)$$
where $a \in \OK^\times$, $u \in \OKt$ with $u(0) \in \OK^\times$ and
$V \in \OKt$ with $t$-adic valuation $1$. Letting $g(t) = u(t)^{-1}$,
the above equation can be rewritten as follows:
\begin{equation}
\label{eq:nonlinearbis}
t(t{-}4a) \cdot z'{}^2 = g(t)^2 \cdot V(z).
\end{equation}
We are going to study how the solution $z$ of the latter differential
equation behaves when $g$ varies.
By Proposition~\ref{prop:solnonlinear}, we know that
Eq.~\eqref{eq:nonlinearbis} has a unique nonzero solution $z_g \in t \Kt$
as soon as $g(0) \neq 0$.
Besides, the proof of Proposition~\ref{prop:solnonlinear} shows that
the $n{+}1$ first coefficients of $z_g$ depend only on the $n$ first
coefficients of $g$. In other words, the association $g \mapsto
z_g$ defines a function $\Omega_n \to t \Kt/(t^{n+1})$ where $\Omega_n$
is the open subset of $\Kt/(t^n)$ consisting of series with nonzero
constant term. In a similar fashion, we notice that $z'_g$ is a
well-defined series in $\Kt/(t^n)$, when $g$ belongs to~$\Omega_n$.

For a given positive integer $n$, we define
$$\begin{array}{rcl}
\varphi_n : \quad \Omega_n & \longrightarrow & \Kt/(t^n) \smallskip \\
g & \mapsto & \displaystyle \frac{z_g}{t(t{-}4a)}
\end{array}$$
(we note that $t{-}4a$ is invertible in $\Kt/(t^n)$).
It follows from the proof of Proposition~\ref{prop:solnonlinear} that
$\varphi_n$ is a polynomial function in $g(0)^{-1}$ and the coefficients
of $g$; in particular, it is locally analytic.

\begin{prop}
\label{prop:dphin}
For $g \in \Omega_n$, the differential of $\varphi_n$ at $g$ is
the function
$$\begin{array}{rcl}
d\varphi_n(g) : \quad \Kt/(t^n) & \longrightarrow & \Kt/(t^n) \\
\deltag & \mapsto &
z'_g \cdot g^{-1} \cdot \psi_{+,n}(\deltag)
\end{array}$$
where $\psi_{+,n}$ is the function defined in Section~\ref{subsec:linear}.
\end{prop}

\begin{proof}
We first differentiate the function $g \mapsto z_g$. This
amounts to find a quantity $\deltaz$ varying linearly with respect
to $\deltag$ for which the relation $z_{g + \deltag} = z_g + \deltaz$
holds at order $1$. Coming back to the definitions, we are led to
the identity
$$t(t{-}4a) \cdot (z'_g + \deltaz')^2 = (g(t)+\deltag(t))^2 \cdot V(z_g
+ \deltaz)
\,+\, \text{higher order terms}.$$
Expanding this relation, we obtain the following linear differential
equation in $\deltaz$:
\begin{equation}
\label{eq:deltaz}
2 t(t{-}4a) \:z'_g \cdot \deltaz' =
 2 g(t) \:\deltag(t) \:V(z_g) + g(t)^2 \:V'(z_g) \cdot \deltaz.
\end{equation}
Observe that $z_g$ and $g$ are both invertible in $\Kt/(t^n)$. We can
then write $\deltaz = z'_g \: g^{-1} \cdot y$ for some $y \in \Kt/(t^n)$.
Performing this change of functions and making use of the relations
\begin{align*}
t(t{-}4a) (z'_g)^2 & = g^2 \: V(z_g) \\
2t(t{-}4a) z'_g z''_g + 2(t{-}2a) (z'_g)^2
& = g^2 z'_g \: V'(z_g) + 2 g g' \: V(z_g)
\end{align*}
(the second one being obtained from the first one by derivation),
Eq.~\eqref{eq:deltaz} becomes
$$t(t{-}4a) y' - (t{-}2a) y = t(t{-}4a) \cdot \deltag.$$
Therefore $y = \psi_{-,n+1}\big(t(t{-}4a) \cdot \deltag\big) =
t(t{-}4a) \cdot \psi_{+,n}(\deltag)$ thanks to Lemma~\ref{lem:psipm}.
We finally derive that
$\deltaz = t(t{-}4a) \cdot z'_g g^{-1} \cdot \psi_{+,n}(\deltag)$
and, simplifying by $t(t{-}4a)$, the proposition follows.
\end{proof}

We now need to introduce norms on $\Kt/(t^n)$.
In order to avoid confusions, we set $E_n = F_n = \Kt/(t^n)$ and use
$E_n$ (resp. $F_n$) for the domain (resp. the codomain) of our
functions. Then, for example, $\Omega_n$ will be considered as
a subset of $E_n$ and $\varphi_n$ as a function from $\Omega_n$
to $F_n$. Similarly $d\varphi_n(g)$ will be viewed as an element
of $\Hom(E_n, F_n)$.

We endow $F_n$ with the usual Gauss norm
$$\big\Vert a_0 + a_1 x + \cdots + a_{n-1} x^{n-1} \big\Vert_{F_n}
 = \max\big(|a_0|, |a_1|, \ldots, |a_{n-1}|\big).$$
On the contrary, we endow $E_n$ with the norm $\Vert f \Vert_{E_n}
= \Vert \psi_{+,n}(f) \Vert_{F_n}$. It is then clear that
$\psi_{+,n} : E_n \to F_n$ is an isometry.

\begin{lem}
\label{lem:isometry}
Let $g \in \OKt/(t^n)$. We assume $z_g \in \OKt/(t^n)$ and that both
$g(0)$ and $z'_g(0)$ are invertible in $\OK$.
Then $d\varphi_n(g) : E_n \to F_n$ is an isometry.
\end{lem}

\begin{proof}
The assumptions ensure that $g$ and $z'_g$ are invertible in
$\OKt/(t^n)$. Therefore the multiplication by $z'_g\: g^{-1}$
is an isometry of $F_n$. The lemma then follows from the
explicit formula of $d\varphi_n(g)$ given by
Proposition~\ref{prop:dphin}.
\end{proof}

We now fix a series $g \in \OKt/(t^n)$ satisfying the assumptions
of Lemma~\ref{lem:isometry}. We define $W_n$ as the open subset
of $E_n$ consisting of series $\gamma$ for which $g + \gamma \in
\Omega_n$. We introduce the two following functions:
$$\begin{array}{r@{\hspace{0.5ex}}c@{\hspace{0.5ex}}lcl}
\multicolumn{3}{r}{\theta_n : \quad W_n}
 & \longrightarrow & F_n \\
\multicolumn{3}{r}{\gamma}
 & \mapsto & \varphi_n(g + \gamma) - \varphi_n(g),
\bigskip \\
\tau _n : \quad  F_n &\times& W_n
 &\longrightarrow&  \Hom({E_n},{F_n} ) \smallskip \\
(\zeta&, & \gamma) & \mapsto & \displaystyle
\left(\deltag \:\mapsto \:
\frac{z'_g + t(t{-}4a) \zeta' + 2(t{-}2a) \zeta}{g+\gamma}
\cdot \psi_{+,n}(\deltag)\right)
\end{array}$$
It follows from Proposition~\ref{prop:dphin}
that $d\theta _n = \tau _n \circ (\theta _n , \Id)$, where $\Id$ is the
identity map on $W_n$.
We can associate to any (locally) analytic function $f$, the numerical
function
$\Lambda (f):\: \mathbb{R} \cup \{\infty\} \longrightarrow \mathbb{R} \cup
\{\infty\}$ defined by
  \begin{equation}\label{eq:defLambda}
 \Lambda (f) (x) = \left \{ \begin{array}{c @{ $ $ } l}
\log \left( \underset{\gamma \in B_{_{E_n}}(e^x)} \sup  \Vert f(\gamma) \Vert \right) & \text{if }f \text{ is defined over } B_{E_n}(e^x)\,,\medskip\\
\infty & \text{elsewise}\,.
\end{array} \right.
\end{equation}
In this equation and in Proposition~\ref{prop:optimalprecision},
$B_{E_n}(\delta)$ (resp. $B_{F_n}(\delta)$) stands for the closed ball in
$E_n$ (resp. in $F_n$) of centre $0$ and radius $\delta$.
We define
\begin{math}
  \Lambda (f)_{\geq 2} (x) = \underset{y \geq 0}{\Inf} (\Lambda (f)(x+y) - 2\,y)
\end{math}
too.

\begin{lem}
\label{lem:Lambda}
Suppose $x < - \log 4$, then
$\Lambda (\theta _n) _{\geq 2} (x)< x$.
\end{lem}

\begin{proof}
  For all $x > 0$, one easily checks that $\Lambda (\Id)(x) = x$ and
  $\Lambda(\tau_n)(x) \geq 0$.  Applying~\cite[Proposition~2.5]{carova15}, we
  obtain $\Lambda(\theta)_{\geq 2} (x) \leq 2\,(x + \log 2 )$ for
  $x \leq - \log 2$.  In particular, $\Lambda (\theta _n) _{\geq 2} (x)< x$
  for $x < - \log 4$.
\end{proof}

We can now state the following important result, that gives the best
possible precision that one can expect for $\varphi_n(g)$ when $g$ is
itself only known up to some finite precision.

\begin{prop}
\label{prop:optimalprecision}
Under the assumption of Lemma~\ref{lem:isometry}, we have
\begin{displaymath}
  \varphi_n \big(g + B_{E_n}(\delta) \big) = \varphi_n (g) + B_{F_n}(\delta)%
  \ \text{ for all }\ %
  \delta < 1/4.
\end{displaymath}

\end{prop}

\begin{proof}
  Applying~\cite[Proposition~3.12]{carova14} with the bound of
  Lemma~\ref{lem:Lambda}, we obtain for $\delta < 1/4$
  $$\varphi _n \big(g + B_{E_n}(\delta ) \big) = {\varphi _n} (g) +
    d\varphi _n(g)\big( B_{E_n}(\delta)\big).$$
  We conclude by applying Lemma~\ref{lem:isometry}.
\end{proof}

\begin{remark}
\label{rmk:2}
Using that $\varphi_n$ is injective, we see that
Proposition~\ref{prop:optimalprecision} implies that
$$\Vert \varphi_n(g + \gamma) - \varphi_n(g) \Vert_{F_n}
= \Vert \gamma \Vert_{E_n}$$
as soon as $g$ satisfies the assumption of Lemma~\ref{lem:isometry}
and $\Vert \gamma \Vert_{E_n} \leq 1/4$.
More generally, applying this result with $g$ replaced with $g +
\delta$ with $\Vert \delta \Vert_{E_n} \leq 1/4$, we find
that $\varphi_n$ is an isometry in restriction to the ball of
centre $g$ and radius $1/4$.
\end{remark}

\noindent
\begin{proof}[Correctness proof of Theorem~\ref{thm:stabnonlinear}.]
  Let $U$, $V$ and $n$ be the input of Algorithm~\ref{algo:IsoSolve}.  We
  first claim that the output of Algorithm~\ref{algo:DiffSolve} satisfies
\begin{equation}
\label{eq:algoproof}
 W(t, q_n) \equiv  u^2 \cdot z'_n{}^{\!2}  \pmod {t^{n-1}, 2^M}.
\end{equation}
for all $n$. We shall prove it by induction on $n$.
The case $n=1$ is easy. Let $m \geq 1$ be a positive integer and $n = 2m-1$. We suppose that Eq.~\eqref{eq:algoproof} is true for $m$.
We set $\lambda = ba^{-1}$. Since $z_m \equiv \lambda t + t(t{-}4a)q_m \pmod {t^n}$, we derive the following relation
\begin{equation}
\label{eq1:proof}
V(z_m)\equiv  t(t{-}4a) \cdot W(t, q_m) \pmod {t^n}.
\end{equation}
Taking the logarithmic derivative of $W$
with respect to $x$, we get
\begin{align*}
\frac{\partial W}{\partial x}(t, q_m)
& = W(t, q_m) \cdot \left(\frac{t{-}4a}{\lambda  + (t{-}4a)q_m}
  + \frac{t}{\lambda  + tq_m}
  + \frac{2t(t{-}4a)\, v' \big (\lambda t + t(t{-} 4a) q_m \big ) }{ v\big (\lambda t + t(t{-} 4a) q_m \big )}\right) \\
& = W(t, q_m) \cdot \left(\frac{t(t{-}4a)}{z_m} + \frac{t(t{-}4a)}{z_m{-}4b} + \frac{2t(t{-}4a)\, v'(z_m)}{v(z_m)} \right).
\end{align*}
Using now Eq.~\eqref{eq1:proof}, we get
$$\frac{\partial W}{\partial x}(t, q_m)
\equiv ( z_m{-}4b) \, v(z_m)^2 + z_m \, v(z_m)^2 + 2 \: v(z_m) \, v'(z_m)
\equiv V'(z_m) \pmod {t^{n-1}}\,.$$
In addition, we have
$q_n \equiv q_m + z'_m u \cdot y_n \pmod {t^n, 2^M}$ by construction.
By Remark~\ref{rem:lindiffsolveval} combined with the induction
hypothesis, we know moreover the $m$ first coefficients of $y_n$
vanish, \emph{i.e.} $y_n \equiv 0 \pmod{t^m, 2^M}$. We then deduce
\begin{align}
W(t, q_n)
 & \equiv W(t, q_m) + z'_m u \: y_n \cdot \frac{\partial W}{\partial x}(t, q_m)
   \pmod {t^{n-1}, 2^M} \nonumber \\
 & \equiv W(t, q_m) + z'_m u \: y_n \cdot V'(z_m)
   \pmod {t^{n-1}, 2^M}. \label{eq:cong1}
\end{align}
Besides, by definition of $y_n$, we have the relation
\begin{displaymath}
  t(t{-}4a) y'_n + (t{-}2a) y_n \equiv
  \frac 1{2u^3} \left( \frac{W(t,q_m)}{z'_m{}^{\!2}} - u^2\right)
  \pmod {t^{n-1}, 2^M}\,,
\end{displaymath}
from which we derive
\begin{equation}
\label{eq:cong2}
W(t,q_m) - u^2 z'_m{}^{\!2} \equiv
2 u^3  z'_m{}^{\!2} \cdot \big(t(t{-}4a) y'_n + (t{-}2a) y_n\big)
\pmod {t^{n-1}, 2^M}\,.
\end{equation}
Similarly, using the congruence
$z_n \equiv z_m + t(t{-}4a) z'_m u \: y_n \pmod {t^n, 2^M}$, we
obtain
\begin{equation}
\label{eq:cong3}
u^2 \, z'_n{}^{\!2} \equiv
u^2 \, z'_m{}^{\!2} + 2 u^2 \, z'_m \, \big( t(t{-}4a) z'_m u \: y_n \big )'
\pmod {t^{n-1},2^M}.
\end{equation}
Combining Eqs.~\eqref{eq:cong1}, \eqref{eq:cong2} and \eqref{eq:cong3},
we end up with
$$\begin{array}{l}
W(t,q_n) - u^2\,z'_n{}^{\!2} \smallskip \\
\hspace{5ex}\equiv
  u\,y_n \cdot \big ( V(z_m) \big )'
+ 2u^3 \, z'_m {}^{\!2} \big ( t(t{-}4a) \, y'_n + (t{-}2a)\, y_n \big )
- 2u^2 \, z'_m \cdot \big ( t(t{-}4a) z'_m u \, y_n \big )'\,, \smallskip \\
\hspace{5ex}\equiv
  u\, y_n \cdot \big ( V(z_m) - t (t{-}4a) u^2 \, z'_m{}^{\!2}  \big )'
\pmod {t^{n},2^M}\,.
\end{array}$$
Using Eq.~\eqref{eq1:proof}, we finally conclude that
Eq.~\eqref{eq:algoproof} holds true for $n$; our claim is proved.

Multiplying Eq.~\eqref{eq:algoproof} by $g^2$ on both sides, we get
$$g^2 \cdot W(t, q_n) \equiv  z'_n{}^{\!2} \pmod {t^{n-1} , 2^M}\,.$$
We now observe that
$$W(t,q_n) \equiv (\lambda + t q_n)^2 \cdot v(\lambda t + t^2 q_n)^2
\pmod 4$$ showing that $W(t,q_n)$ is a square modulo $4$. It thus admits a
square root $w \in \OKt$, up to possibly replacing $K$ by its unique
unramified extension of degree~$2$ (see also Remarks~\ref{rem:HUHV}
and~\ref{rem:HUHVbis}).  We define $g_n = {z'_n}/w$, so that we have
$z_n = t(t{-}4a)\:\varphi_n (g_n)$ and
$\Vert g^2 - g_n^2 \Vert _{F_{n-1}} \leq 2^{-M}$. The last inequality
indicates in particular that $g(0)^2 \equiv g_n(0)^2 \pmod {\pi^{eM}}$. We
normalize $g_n$ in such a way that $g(0) \equiv g_n(0) \pmod 4$ (this is
always possible because $M \geq 3$). Then the series $g + g_n$ is divisible by
$2$ and its constant term has valuation $1$.  As a consequence, $g + g_n$ is
invertible in $\Kt$ and its Gauss norm is $1/2$. We deduce that
$$
\Vert g - g_n \Vert_{F_{n-1}}
= \Vert g^2 - g_n^2  \Vert_{F_{n-1}} \cdot
\Vert g + g_n \Vert_{F_{n-1}}^{-1}
= 2 \cdot \Vert g^2 - g_n^2  \Vert_{F_{n-1}} \leq 2^{-M+1}.$$
So $\Vert g - g_n \Vert _{E_{n-1}} \leq 2^{-N}$.
Using Remark~\ref{rmk:2}, we conclude that
$$\Vert z_g - z_n \Vert _{F_n} \leq \Vert \varphi _n (g) - \varphi _n (g_n) \Vert _{F_{n-1}} = \Vert g - g_n \Vert _ {E_{n-1}}\leq 2^{-N}.$$

We finally justify that all computations stay within $\OK$, so that no
error is raised during the execution of \texttt{IsoSolve}.
Examining the successive operations performed by
the algorithm, we see that nonintegral coefficients may show up only
during the computation of $f_n$ (because of the division by $2$) and
that of $y_n$ (because of the call to \texttt{LinDiffSolve}).
After Proposition~\ref{prop:stablinear}, we are reduced to check that
$f_n$ and $y_n$ have integral coefficients modulo $t^n$.
By construction, they are related by the relation
$$t(t{-}4a) y'_n + 2(t{-}2a) y_n \equiv f_n \pmod{t^n}$$
so the integrality of $y_n$ will directly imply that of $f_n$.  By
construction, $z_n \equiv z_m + z'_m u y_n \pmod{t^n}$.  Besides, we know that
$z_m$ and $z_n$ have integral coefficients. We deduce that $y_n$ has integral
coefficients as well, given that $z'_n$ and $u$ are invertible in $\OKt$.
\end{proof}

\subsection{Experiments}
\label{sec:experiments}

We made an implementation of both Algorithm~\ref{algo:LinSolve} and the Pad\'e
approximant step (thanks to the \textsc{half-gcd} algorithm given
in~\cite{thome03}) with the \textsc{magma} computer algebra
system~\cite{magma}. Our implementation is available at~\cite{github}; it is
fairly optimized and can compute isogenies up to degree $10^6$ in less than
one minute (see precise timings on Figure~\ref{fig:timingsgf2},
page~\pageref{fig:timingsgf2}).  The degree $11$ toy example presented below
was computed with this software as well.

\subsubsection{A toy example}
\label{sec:a-toy-example}

We consider the elliptic curve given by
\begin{math}
  E/\FF:\, y^2 + x\,y = x^3 + 1\,.
\end{math}
The abstract structure of its endomorphism ring is the ring of integers of
$\QQ(\sqrt{-7})$, the class group of which is trivial. In particular, there exists
an isogeny of degree 11, which turns out to be an endomorphism of $E$. Let us
compute it.

We first lift $E$ over $\QQ$ as
${\mathcal E}/\QQ:\, y^2 = x^3 + 2^{-2}\,x^2 + 1 + O(2^{9})$. Using
computations in $\Q(\sqrt{-7})$ (as detailed in Section~\ref{sec:isog-large-degr}),
we find that ${\mathcal E}/\QQ$ is $11$-isogenous to the curve ${\mathcal
E'}/\QQ:\, y^2 = x^3 + 2^{-2}\,x^2 + 225 + O(2^{9})$, the ``differential
constant'' of the isogeny being equal to $41 + O(2^{9})$.
A simple Newton iteration leads to $4a = -16 + O(2^{9})$ and
$u^2 = 65 -16t + 4t^2 + O(2^{9})$. Extracting the inverse square root, we
obtain
\begin{small}
  \begin{multline*}
    \frac{1}{u} = 225 - 248\,t - 226\,t^2 + 208\,t^3 - 122\,t^4 + 240\,t^5 + 172\,t^6 + 160\,t^7 -
    250\,t^8 - 80\,t^9 - 60\,t^{10} + 96\,t^{11} +
    O\big(2^{9}, t^{12}\big),
  \end{multline*}
\end{small}\noindent
from which it is easy to compute $u$ and $u^{-3}$.
All precomputations of Algorithm~\ref{algo:DiffSolve} are now finished and
we can start the first step of the main Newton iteration. We begin with
$q_0 = 10 + O\big(2^{9},t)$ and find
\begin{align*}
  z_0 &= 41\,t + 10\,t^2 + O(2^{9}, t^3)\,,\
  & s_0 &= 164\,t + O(2^{9}, t^2)\,,\
  & r_0 &= 113 + 152\,t + O(2^{9}, t^2)\,,\\
  f_0 &= 228\,t + O(2^{9}, t^2)\,,\
  & y_0 &= -211\,t + O(2^{9}, t^2)\,,\
  & q_1 &= 10 - 43\,t + O(2^{9}, t^2)\,.
\end{align*}
Three intermediary steps follow similarly, allowing to increase
$t$-adic precision from $O(t^2)$ to $O(t^3)$, to $O(t^6)$ and then
to $O(t^{12})$. After these computations, we are left with
\begin{small}
  \begin{multline*}
    q_{11} = 10 - 43\,t + 140\,t^2 - 6\,t^3 + 182\,t^4 - 89\,t^5 + 228\,t^6 + 246\,t^7 +
    248\,t^8 + 76\,t^9 + 20\,t^{10} + 206\,t^{11} +
    O\big(2^{9},t^{12}\big)
  \end{multline*}
\end{small}%
and a last iteration finally yields
\begin{small}
$$\begin{array}{l}
    z_{11} = 41\,t + 94\,t^2 + 5\,t^3 + 116\,t^4 + 210\,t^5 + 82\,t^6 - 201\,t^7 + 188\,t^8 +\\
    \hspace{30ex} 214\,t^9 + 40\,t^{10} + 156\,t^{11} - 180\,t^{12} + O(2^{9},t^{13})\,, \smallskip \\
    s_{11} = 200\,t + 48\,t^2 - 4\,t^3 - 32\,t^4 + 224\,t^5 - 128\,t^6 + 32\,t^7 + 160\,t^8 +\\
    \hspace{30ex} 96\,t^9 - 192\,t^{10} + 96\,t^{11} - 128\,t^{12} + O(2^{9},t^{13})\,, \smallskip \\
    r_{11} = 113 + 200\,t - 222\,t^2 - 136\,t^3 + 175\,t^4 - 56\,t^5 - 10\,t^6 + 48\,t^7 - 137\,t^8 +\\
    \hspace{30ex} 168\,t^9 + 226\,t^{10} - 240\,t^{11} + 238\,t^{12}  + O(2^{9},t^{13})\,, \smallskip \\
    f_{11} =  -184\,t + 48\,t^2 - 100\,t^3 - 32\,t^4 + 184\,t^5 + 16\,t^6 - 4\,t^7 - 192\,t^8 -\\
    \hspace{30ex} 24\,t^9 + 16\,t^{10} + 180\,t^{11} + 256\,t^{12}  + O(2^{9},t^{13})\,, \smallskip \\
    y_{11} =  94\,t^{12} + 131\,t^{13} - 172\,t^{14} - 82\,t^{15} + 34\,t^{16} - 215\,t^{17} + 80\,t^{18} -
    120\,t^{19} +\\
    \hspace{30ex} 70\,t^{20} - 233\,t^{21} + 110\,t^{22} + 161\,t^{23} + O(2^{9},t^{24})\,,
\end{array}$$
\end{small}%
then
\begin{small}
  \begin{multline*}
    q_{23} = 10 - 43\,t + 140\,t^2 - 6\,t^3 + 182\,t^4 - 89\,t^5 + 228\,t^6 +
    246\,t^7 + 248\,t^8 + 76\,t^{9} + 20\,t^{10} + 206\,t^{11} + 206\,t^{12} +
    243\,t^{13}\\\ \ \ \ \ \ \ \ - 210\,t^{14} - 143\,t^{15} - 206\,t^{16} + 145\,t^{17} +
    244\,t^{18} - 218\,t^{19} + 10\,t^{20} + 137\,t^{21} - 166\,t^{22} +
    147\,t^{23} + O(2^{9},t^{24})\,,
  \end{multline*}
\end{small}%
and
\begin{small}
  \begin{multline*}
    z_{24} = 41\,t + 94\,t^2 + 5\,t^3 + 116\,t^4 + 210\,t^5 + 82\,t^6 -
    201\,t^7 + 188\,t^8 + 214\,t^9 + 40\,t^{10} + 156\,t^{11} - 180\,t^{12} +
    6\,t^{13}\\\ \ \ \ \ \ \ \ - 102\,t^{14} - 85\,t^{15} - 14\,t^{16} + 57\,t^{17} +
    118\,t^{18} + 97\,t^{19} - 116\,t^{20} - 178\,t^{21} - 210\,t^{22} -
    15\,t^{23} + 166\,t^{24} + O(2^{9},t^{25})\,.
  \end{multline*}
\end{small}%
A call to the \textsc{half-gcd} algorithm with input $\sqrt{z_{24}\, /\, t} \mod 2$,
which is $1 + t + t^3 + t^{7} + t^{8} + t^{9} + t^{11} + O(t^{12})$, allows us
to recover the rational function
\begin{displaymath}
  \frac{t^5 + t^3 + t^2 + t + 1}{t^5 + t^4 + t^3 + t^2 + 1}+O(t^{12})\,,
\end{displaymath}
from which one deduces that the curve $E/\FF$ is self $11$-isogenous
under the mapping
$$x \mapsto
\frac{x\,(x^5+ x^3 + x^2 + x +1)^2}{(x^5 + x^4 + x^3 + x^2 + 1)^2}\,.$$

\subsubsection{Some timings}
\label{sec:some-timings}

We made use of our \textsc{magma} software to measure the time needed to
compute isogenies up to degree $1\,500\,000$ for an elliptic curve defined
over $\FF$. Results are reported on Figure~\ref{fig:timingsgf2}.  Since
multiplying $2$-adic series can be done in almost linear times with
\textsc{magma}, the time complexity of our implementation is almost linear as
well: the observed timings fit rather well with the expected time complexity,
which is $O(\ell\,\log^2 \ell)$.  The timings for \textsc{half-gcd} are
significantly smaller (by a factor close to $3$) because of two facts: first,
the degree of the inputs is $2$ times smaller than in
Algorithm~\ref{algo:LinSolve} and, second, the underlying polynomial
arithmetic over $\FF$ is slightly more efficient than the arithmetic with
$2$-adic series in \textsc{magma}.

\begin{figure}[htbp]
  \centering
  \includegraphics[width=0.55\textwidth, angle=-90]{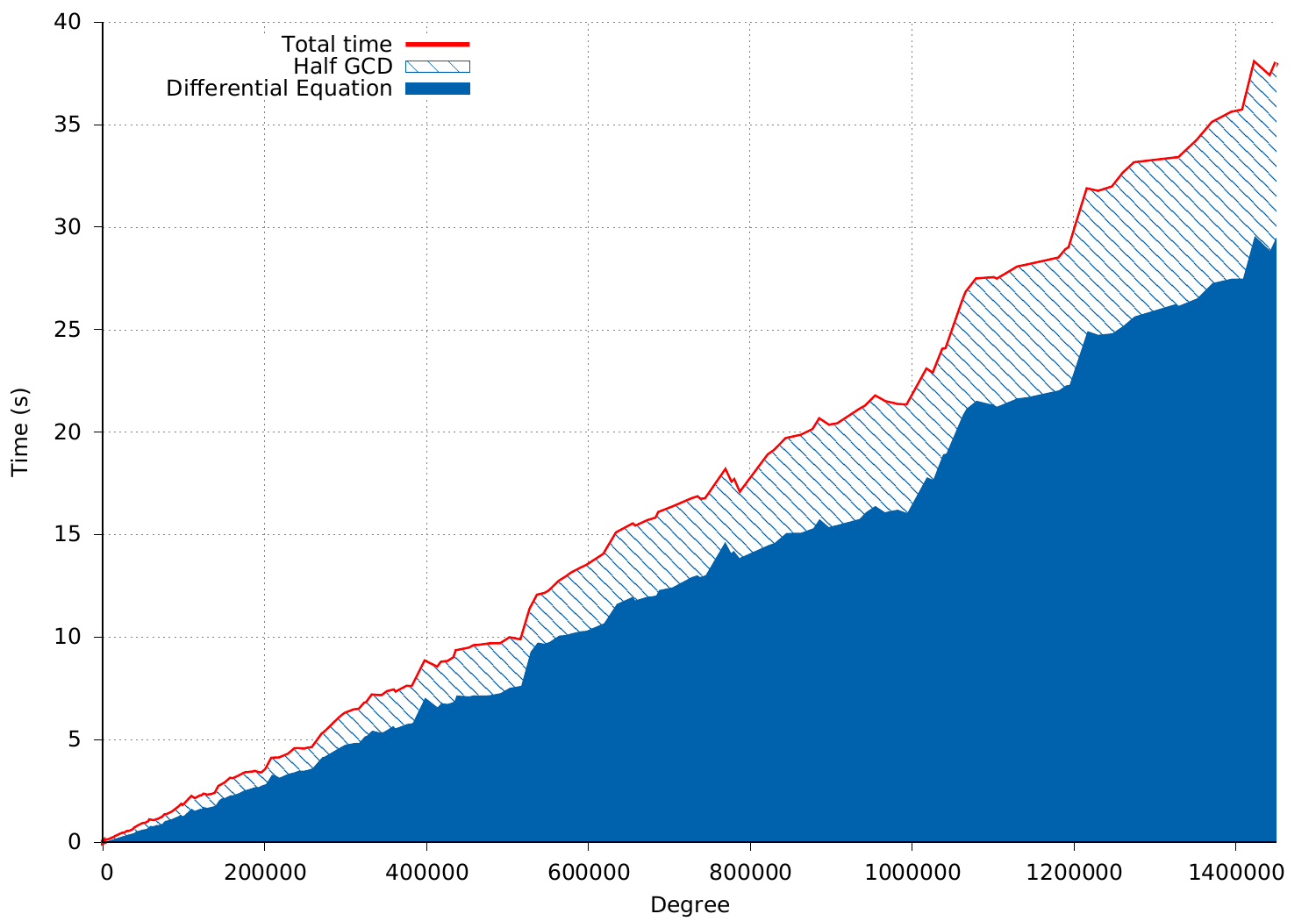}

  \medskip

  {\scriptsize%
  Timings obtained with \textsc{magma
  v2.24-10} on a laptop with an \textsc{intel} processor
  \textsc{i7-8850h@2.60ghz}}

  \caption{Isogeny computations in $\FF$.}
  \label{fig:timingsgf2}
\end{figure}

\section{Applications}
\label{sec:applications}

Thanks to the results of~\cite{bomosasc08,lesi08,lava16} in odd characteristic
and the case of characteristic 2 being solved in
Section~\ref{sec:main-result}, we now have in all characteristic fast
algorithms for computing isogenies, at least if we have a Weierstrass model of
the isogenous curve and the isogeny differential.
In this section, we are interested in the calculation of irreducible
polynomials. We show how we can extend to the case of very small finite
fields, especially $\GF{2}$, the construction of~\cite{CL13}.
With this aim, we start with a brief presentation on endomorphism rings and
isogenies in Section~\ref{sec:endom-ring-isog}, and show in
Section~\ref{sec:isog-large-degr} how to calculate the isogenous curves and
the isogeny differentials over finite fields of small characteristic. Then, in
Section~\ref{sec:fast-comp-irred}, we apply this construction to build
irreducible polynomials and we end with an example in
Section~\ref{sec:an-example}.  \medskip

\subsection{Endomorphism ring and isogenies}
\label{sec:endom-ring-isog}

We briefly introduce some facts about the theory of complex
multiplication. Good references are \cite{lang73,silv2,cox13}.

An isogeny $E_1\rightarrow E_2$ of elliptic curves defined over a field $k$ is
a surjective morphism of curves that induces a group homomorphism
$E_1(\bar{k})\rightarrow E_2(\bar{k})$. We denote by $\Hom_k(E_1, E_2)$ the
set of homomorphisms from $E_1$ to $E_2$ over $k$ and let
$\End_k(E) = \Hom_k(E, E)$. We write
$\End(E) = \End_{\bar k}(E)$. Composition of endomorphisms gives a ring
structure on $\mathcal{O} = \End(E)$, and we refer to $\mathcal{O}$ as the
ring of endomorphisms of $E$.

As a $\Z$-module, $\End_k(E)$ is free of rank at most four. More precisely,
$\End_k(E)$ is either $\Z$, an order in an imaginary quadratic field (ordinary
case) or an order in a definite quaternion algebra (supersingular case). By
definition, ordinary or supersingular elliptic curves have complex
multiplication. Moreover, every endomorphism $\varphi$ satisfies in $\End_k(E)$ a
quadratic characteristic polynomial with integer coefficients,
\begin{math}
  \varphi^2 - t\,\varphi + d = 0\,.
\end{math}
The integer $t$, denoted $\Tr(\varphi)$, is called the trace of
$\varphi$. Over $k = \GF{q}$, the Frobenius endomorphism $\phi _q$ takes a
leading role, since it determines the group and $\End_k(E)$-structure of the
rational points of $E$~\cite{lenstra96}. In particular, it satisfies the Weil
polynomial
\begin{equation}\label{eq:weil}
\phi_q^2 - t\,\phi_q + q = 0
\end{equation}
where $t = \Tr \phi_q =: \Tr E$ is such that $q+1-t$ is the number of rational
points of $E$ over $k$.\smallskip

What follows is for elliptic curves over $\mathbb C$, but these results reduce
well to elliptic curves over finite fields. In the ordinary case, let
$\mathcal{O}$ be an order in an imaginary quadratic field $\kappa$. Then the theory
of complex multiplication states that there is a number field $L$ containing
$\kappa$ and an elliptic curve $E$ over $L$ with $\End_{\bar L}(E) =
\mathcal{O}$. Let $p$ be a prime that splits completely in $\mathcal{O}_L$ and
$\mathfrak p$ be a prime of $\mathcal{O}_L$ above $p$, so that
$\mathcal{O}_L / {\mathfrak p} \simeq \GF{p}$. If $\mathfrak p$ does not
divide the discriminant of $E$, then $E$ has good reduction modulo
$\mathfrak p$. Let $\bar E$ denote this reduction, then
$\End_{\bar {\mathbb F}_p}(\bar E) \simeq \mathcal{O}$. Conversely, every
elliptic curve $\bar E$ arises as the reduction of an elliptic curve $E$ over
some $L$ with same ring of endomorphisms $\mathcal O$, called the canonical
lift of $\bar E$.

There is a one-to-one correspondence between the isomorphism classes of
elliptic curves $E/\GF{q}$ with $\End(E)=\mathcal O$ and the ideal class group
$\Cl(\mathcal{O})$ (\textit{i.e.} the quotient group of the fractional
$\mathcal O$-ideals that are prime to the conductor $\mathcal O$ by its
subgroup of principal ideals). To an invertible $\mathcal O$-ideal
$\mathfrak a$, one associates the elliptic curve
$E_{\mathfrak{a}} \simeq \mathbb{C}/\mathfrak{a}$.  An ideal $\mathfrak a'$ is equivalent to
$\mathfrak a$ in $\Cl(\mathcal{O})$ if and only if $\mathbb{C}/\mathfrak{a}'$
is isomorphic to $E_{\mathfrak{a}}$.

Let now $\mathfrak{l}$ be an invertible $\mathcal O$-ideal, and define the
the kernel of $\mathfrak{l}$ in $E_{\mathfrak{a}}$ to be the intersection of the kernels of all
endomorphisms in $\mathfrak{l}$. We denote it $E_{\mathfrak{a}}[\mathfrak{l}]$,
\begin{eqnarray*}
  E_{\mathfrak{a}}[\mathfrak{l}] &\simeq& \{ z \in \mathbb{C} : \alpha z \in \mathfrak{a},\ \text{for
                           all}\ \alpha \in \mathfrak{l} \subset \End(E_{\mathfrak{a}}) \}\,,\\
  &\simeq& \mathfrak{l}^{-1}\mathfrak{a}/\mathfrak{a}.
\end{eqnarray*}
The identity map on $\mathbb{C}$ induces the isogeny
$I : \mathbb{C}/\mathfrak{a} \rightarrow \mathbb{C}/\mathfrak{l}^{-1}\mathfrak{a}$
with kernel $E_{\mathfrak{a}}[\mathfrak{l}]$. The norm of $\mathfrak{l}$ is equal to the
degree of the isogeny. The terminology ``quotient isogeny'' is sometimes used,
together with the notation $I : E_{\mathfrak{a}} \rightarrow E_{\mathfrak{a}} / E_{\mathfrak{a}}[\mathfrak{l}]$.

Every isogeny between elliptic curves with isomorphic ring of endomorphisms
arises in this way.
In particular, let $I_1 : E_1 \rightarrow E_2$ be a first isogeny defined by
$\mathfrak{l_1}$, and $I_2 : E_2 \rightarrow E_3$ be a second isogeny
defined by $\mathfrak{l_2}$, then the kernel of
\begin{math}
  I_2 \circ I_1 : \xymatrix{ E_1 \ar@{->}[r]^{I_1} & E_2 \ar@{->}[r]^{I_2} & E_3}
\end{math}
is $E_1[\mathfrak{l}_1 \mathfrak{l}_2 ]$.\smallskip

These facts are well summarized in the following proposition.
\begin{prop}[{\cite[Prop. 1.2, Chap. II]{silv2}}]
  \ %
  \begin{enumerate}
  \item[(a)] Let $E_{\mathfrak a}$ be an elliptic curve with endomorphism ring
    ${\mathcal O}_\kappa$ and let $\mathfrak l$ and $\mathfrak{l}'$ be
    non-zero fractional ideals of ${\mathcal O}_\kappa$.
    \begin{enumerate}
    \item[(i)] $\mathfrak{l}\,\mathfrak{a}$ is a lattice in $ \mathbb{C}$.
    \item[(ii)] The elliptic curve $E_{\,\mathfrak{l}\,\mathfrak{a}}$ satisfies
      $\End(E_{\,\mathfrak{l}\,\mathfrak{a}}) \simeq {\mathcal O}_\kappa$.
    \item[(iii)]
      $E_{\,\mathfrak{l}\,\mathfrak{a}} \simeq E_{\,\mathfrak{l}'\,\mathfrak{a}}$
      if and only if $\mathfrak{l} \simeq \mathfrak{l}'$ in
      $\Cl(\mathcal{O_\kappa})$.
    \end{enumerate}
    Hence there is a well-defined action of $\Cl(\mathcal{O_\kappa})$ on the
    set of elliptic curves with endomorphism ring ${\mathcal O}_\kappa$
    determined by
    \begin{math}
      \mathfrak{l} * E_{\mathfrak{a}} = E_{\,\mathfrak{l}^{-1}\,\mathfrak{a}}\,.
    \end{math}
  \item[(b)] The action of $\Cl(\mathcal{O_\kappa})$ described in (a) is
    simply transitive. In particular $\# \Cl(\mathcal{O_\kappa})$ is equal to
    the number of elliptic curves with endomorphism ring
    ${\mathcal O}_\kappa$.
\end{enumerate}

\end{prop}

\subsection{Isogenies of large degree}
\label{sec:isog-large-degr}

Let $E$ be an elliptic curve with complex multiplication defined over a finite
field $k$ and $\ell>2$ a large prime integer. Here, the field of definition of $E$
is supposed to be very small compared to $\ell$ so that in this case, up to an
endomorphism, an $\ell$-isogeny can be written as a composition of small
isogenies. This can be done by working in the ideal class group of the
endomorphism ring of $E$. In fact, the situation is very similar to that
behind the algorithm given by Kohel in his thesis for computing the
endomorphism ring of an elliptic curve.

\begin{thm}[{\cite[Th.~1]{kohel96}}]\label{th:kohel}
  There exists a deterministic algorithm that, given an elliptic curve $E$ over
  a finite field $k$ of $q$ elements, computes the isomorphism type of the
  endomorphism ring of $E$ and if a certain generalization of the Riemann
  hypothesis holds true, for any $\varepsilon>0$ runs in time
  $O(q^{1/3+\varepsilon})$.
\end{thm}

Since in our case of interest, the field of definition of the curves is rather
small while the degrees of the isogenies are rather large, we can suppose that
we are given the endomorphism ring $\End(E)$ of $E$ as an order
$\mathcal{O}_\kappa$ in an imaginary quadratic field
$\kappa=\Q(\sqrt{-\Delta})$, $\Delta$ a primitive discriminant. For the sake
of simplicity, we assume that this order is maximal, \textit{i.e.}  equal to
the ring of integers $\Z[\omega]$ of $\kappa$.%

In this context, a prime integer $\ell\neq p$ that splits in $\kappa$ is
usually called an Elkies prime for $E$. We more generally define \emph{Elkies
  degrees} for $E$ as integers whose prime divisors are all Elkies primes for
$E$. Incidentally, there exist a $k$-rational $\ell$-isogeny from $E/k$ to another
curve $\tilde{E}/k$.
This said, the computation of the isogenous curve $\tilde{E}$ reduces
to calculations in the ideal class group of $\mathcal{O}_\kappa$.\smallskip

More precisely, let $m$ be one of the Elkies prime divisors of $\ell$.
Let
$\Cl(\mathcal{O}_\kappa)
$ be the ideal class group associated to
$\mathcal{O}_\kappa$\,.
We have
\begin{math}
  \vert \text{Cl}(\mathcal{O}_\kappa) \vert = O( \sqrt{ \vert \Delta \vert}
  )\,.
\end{math}
In addition, every ideal class in $\Cl(\mathcal{O}_\kappa)$ contains an
ideal of norm less than $\sqrt{\vert \Delta \vert }$. So,
$\Cl(\mathcal{O}_\kappa)$ is generated by classes of ideals of norm less
than $\sqrt{\vert \Delta \vert }$. Let $\mathfrak{m}$ be an ideal
$(m, a_\mathfrak{m}+b_\mathfrak{m}\,\omega)$ in $\mathcal{O}_\kappa$ that divides
$(m)$,
and write
\begin{math}
  \mathfrak{u}_1\: \mathfrak{m} = \mathfrak{u}_2 \: \prod \limits _{i=1} ^h {\mathfrak{p}_i ^{e_i}}\,,
\end{math}
where $\Norm(\mathfrak{p}_i) \leq \sqrt{\Delta}$ and $\mathfrak{u}_1$ and
$\mathfrak{u}_2$ are two principal ideals. Each prime ideal $\mathfrak{p}_i$
determines an isogeny, and $\mathfrak{u}_1$, $\mathfrak{u}_2$ correspond to
endomorphisms. Their product yields the isogeny defined by the
following chain of small degree isogenies,
\begin{multline}\label{eq:18}
  E \longrightarrow \bigslant{E} { E[\mathfrak{p}_1]} \longrightarrow
  \bigslant{E} { E[\mathfrak{p}_1^2]} \longrightarrow \ldots \longrightarrow
  \bigslant{E} { E[\mathfrak{p}_1^{e_1}]} \longrightarrow\\
  \bigslant{E} { E[\mathfrak{p}_1^{e_1}\mathfrak{p}_2]} \longrightarrow
  \bigslant{E} { E[\mathfrak{p}_1^{e_1}\mathfrak{p}_2^2]} \longrightarrow
  \ldots \longrightarrow \bigslant{E} {
    E[\mathfrak{p}_1^{e_1}\mathfrak{p}_2^{e_2}]} \longrightarrow \\
  \ldots\\
  \bigslant{E} {
    E[\mathfrak{p}_1^{e_1}\mathfrak{p}_2^{e_2}\cdots\mathfrak{p}_{h-1}^{e_{h-1}}
    \mathfrak{p}_h]} \longrightarrow \cdots \longrightarrow \bigslant{E}{
    E[\mathfrak{p}_1^{e_1}\mathfrak{p}_2^{e_2}\cdots
    \mathfrak{p}_{h}^{e_{h}}]}\,.
\end{multline}
We arrive in this way at the isogenous curve
$\tilde{E} = {E}\,/\,{ E[\mathfrak{p}_1^{e_1}\,\mathfrak{p}_2^{e_2}\,\cdots
  \,\mathfrak{p}_{h}^{e_{h}}]}$.\medskip

To avoid the divisions by $p$ that do appear in calculations for computing
explicitly all these isogenies, we reconsider Chain~(\ref{eq:18}) from the
standpoint of the canonical lifts of these curves.
Results by Serre and Tate~\cite{LST64} enable to lift canonically $E$ as
$\mathcal{E}/K$ where $K$ is an unramified extension of $\mathbb{Q}_p$ of
degree $n$ such that $k$ is its residue field.

An algorithm for computing canonical lifts at
$p$-adic precision $N$ in time complexity $O(p^2\,N^2)$ up to
some polylogarithmic factors can be for instance found
in~\cite{Satoh00,SST03}. According to Theorem~\ref{thm:stabnonlinear}
and~\cite[Theorem~2]{lava16}, the lifting process has to be done with
$p$-adic precision equal to $O(\lceil \log _{p} (\ell) \rceil)$ in order to be
able to reduce the results modulo $p$.
Since in this situation the principal ideal $(q)$ splits in
$\mathcal{O}_\kappa$ into two prime ideals, one chooses arbitrarily the
residue field given by one of the two, which allows us to embed integers
from $O_\kappa$ into $K$, $\omega \mapsto \sqrt{-\Delta}$.

Now, starting from $\mathcal{E}$ in Chain~\eqref{eq:18}, we use V\'elu's
formulas to compute for each $\mathfrak{p}_i$ a normalized isogenous
curve~\cite{velu71}.  These formulas require that one knows the kernels
$\mathcal{E}[\mathfrak{p}_i]$. This can be easily done for small degrees,
\textit{e.g.} degree 2, by factoring division polynomials. When this approach
(of cubic complexity in the degree) is too expansive, an alternative approach
is to use modular polynomials. It enables to find the $j$-invariant of the
isogenous curve. Motivated by point counting on elliptic curves, Elkies gave
an elegant method to derive from it an explicit normalized equation for this
isogenous curve. This algorithm, of quadratic complexity in the degree, is far
beyond the scope of this paper and we refer to~\cite{Schoof95} for details.
This process yields a $\prod_{i=1} ^h \mathfrak{p}_i^{e_i}$-isogenous curve
$\tilde{\mathcal{E}}$.
Furthermore, the curve $\tilde{\mathcal{E}}$ is also $m$-isogenous to
$\mathcal{E}$ up to the endomorphism
$\mathfrak{u}_2\mathfrak{u}_1^{-1}$. The differential of this $m$-isogeny is
thus equal to the embedding of $\mathfrak{u}_2\,\mathfrak{u}_1^{-1}$ in $K$\,.

We now iterate this construction for every remaining prime divisor $m$ of
$\ell$, counting multiplicity, and go from one isogenous curve to the
next. We arrive in this way to a $\ell$-isogenous curve, and its
differential. Recall that it is faster to compute the isogeny from $\mathcal
E$ to $\tilde{\mathcal{E}}$ than computing all the $\mathfrak{p}_i$-isogenies
and composing them (see Remark~\ref{rmk:1}).
We call Algorithm~\ref{algo:DiffSolve} to find the solution $z(t)$ of
Eq.~\eqref{eq:nonlinear} modulo $t^{2\ell +2}$. It remains to reduce $z(t)$
modulo $p$ and compute its Pad\'e approximant to recover the rational
function that gives the isogeny.\smallskip

The conclusion of this section is that we can compute an $\ell$-isogenous
curve and a rational representation of the isogeny in quasi-linear time in
$\ell$ when $\ell\gg q$.
\begin{thm}
  \label{thm:9}
  Given an ordinary elliptic curve $E$ defined over a finite field $k$ with
  characteristic $p$ and cardinality $q=p^n$ such that its endomorphism ring
  $\mathcal{O}_\kappa$ is maximal and $\ell$ is an Elkies degree for $E$,
  there exists an algorithm that computes an equation of an $\ell$-isogenous
  curve $\tilde{E}$ of $E$ and the isogeny with time complexity
  $O( n\, \ell + p^2 + q^{3/2})$ up to some polylogarithmic factors.
\end{thm}

\noindent
Under reasonable heuristic assumptions detailed in~\cite{BS11}, the $q^{3/2}$
term can be replaced by $\LL{1/2,\sqrt{3}/2}(q)$ where $L$ denotes the usual
subexponential functions
\begin{displaymath}
  \LL{\alpha,{\mathsf c}}\left(x\right) = \exp(\,\left({\mathsf c}+o(1)\right)\,\left(\log
      x\right)^\alpha\,\left(\log\log x\right)^{1-\alpha}\,)\,.
\end{displaymath}

\begin{remark}
  In some cases, the isogeny differential $c$ is rational and the lifting does
  not need to be canonical. For example if $\ell$ is an integer coprime to
  $p$, then the multiplication map $[\ell] \: : E \longrightarrow E$ is a
  separable isogeny that can be computed by lifting arbitrarily the equation
  of the curve $E$ and taking $c = 1 /{\ell}$ in $\mathbb{Q}_q$.
\end{remark}

\subsection{Irreducible polynomials over finite fields}
\label{sec:fast-comp-irred}

Given a finite field $k$, with characteristic $p$ and cardinality $q=p^n$, and
a degree $d$, the Couveignes-Lercier Las Vegas algorithm achieves a notable
quasi-linear asymptotic complexity in $d$ for computing an irreducible
polynomial of degree $d$ over $k$~\cite{CL13}. It is based on elliptic curves
with a number of points which is divisible by prime divisors of $d$. So, these
curves define separable isogenies whose kernels have only rational points. In
its primary form (see Lemma~\ref{lem:cl}), this algorithm yields a highly
efficient method to calculate an irreducible polynomial when $d$ is a prime not
dividing $p\,(q-1)$ such that $4d\leqslant {q^{\frac{1}{4}}}{}$.
For the sake of completeness, we briefly present this construction in
Section~\ref{sec:constr-citecl13}.

However, we note that when $d$ is not a prime or when $d$ is larger than
$q+ 1 + 2\,\sqrt{q}$, \cite{CL13} necessarily involves the use of the
Kedlaya-Umans algorithm~\cite{KU11}.  Unfortunately, this algorithm is widely
considered impractical, and in our case of interest where $q$ is negligible
compared to $d$, especially the important case $k=\GF{2}$, we can no more rely
on this method.

We show in this section that we can adapt the construction to elliptic curves
that do not necessarily have a cardinality divisible by the prime divisors of
$d$. They simply have to admit rational isogenies of degree $\ell$ where
$\ell$ is of the form $p_1^{e_1}$ or of the form $p_1^{e_1}\,p_2^{e_2}$ with
$p_1$ and $p_2$ odd prime integers. Given an elliptic curve, it yields an
infinite dense list of reachable degrees $d$. Except for the few degrees $d$
that can not be written as $d=\phi(\ell)$ or $d=\phi(\ell)/2$, we more
generally have good expectation to find an elliptic curve that may work for a
degree $d$ fixed in advance.
We develop this aspect in Section~\ref{sec:an-extend-algor}.

\subsubsection{Overview of~\cite{CL13}}
\label{sec:constr-citecl13}

Let $I /k: E/k\rightarrow \tilde{E}/k$ be a degree $\ell$ separable isogeny
where $E/k$ and $\tilde{E}/k$ is given by an affine Weierstrass equation in
$x$, $y$. We denote by $O_E$ and $O_{\tilde{E}}$ the points at infinity of $E$
and $\tilde{E}$.

We assume that $\ell$ is a positive odd number and the kernel $\Ker
I$ is cyclic.
Let $T\in E({\bar k})$ be a generator of $\Ker I$.  Let
$\psi_I(x)\in k[x]$ be the degree $(\ell-1)/2$ polynomial
\begin{equation}\label {eq:psiI}
  \psi_I(x) = \prod_{ 1 \leqslant k \leqslant (\ell-1) /2}(x-x(kT))\,.
\end{equation}
There exists a degree $\ell$ polynomial $\phi_I(x)\in k[x]$ such that the
image of the abscissa of the point $(x,y)$ by $I$ is
$\eta(x)={\phi_I(x)}/{\psi_I^2(x)}$.

Now, let $A$ be a $k$-rational point on $\tilde{E}$ such that
$2A\not =O_{\tilde{E}}$ and let $B \in E({\bar k})$ be a point on $E$ such
that $I(B)=A$.  We can define the degree $\ell$ polynomial
\begin{equation*}
  f_{I,A}(x) = \phi_I(x)-x(A)\psi_I^2(x)\in k[x].
\end{equation*}
Its roots are the $x(B+kT)$ for $0 \leqslant k < d$, and they are pairwise
distinct because $2A\not =O_{\tilde{E}}$.  So $f_{I,A}(x)$ is a degree $\ell$ separable
polynomial. Furthermore, it is reducible if and only if the fiber
$I^{-1}(A)$ is.

This happens to be true when $A$ is a point of order $\ell$, and $\ell$ is a
prime not dividing $p\,(q-1)$ such that it divides exactly the number
$q+1-\Tr E$ of rational points of $E$. In this case, the Weil
polynomial~\eqref{eq:weil} splits as
\begin{math}
  (X - r)\,(X- s) \bmod \ell
\end{math}
such that $r=1 \bmod \ell$ and $s=q \bmod \ell$. Especially the Galois
orbit of $B$ has cardinality $\ell$ since $\phi_q(B) = r\,B$ and the order
of $r$ is $\ell$ in $(\Z/\ell^2\Z)^\times$~\cite[Section 4.2]{CL13}.

When $\ell$ is small enough, we can find with high confidence such an elliptic
curve $E$. All in all, it yields a quite efficient method to compute an
irreducible polynomial.

\begin{lem}[{\cite[Lemma 6 ($\delta=1$)]{CL13}}]
  \label{lem:cl}
  There exists a probabilistic (Las Vegas) algorithm that on input a finite
  field $k$ with characteristic $p$ and cardinality $q=p^n$, a prime integer
  $\ell$ not dividing $p\,(q-1)$ such that $4\ell\leqslant {q^{\frac{1}{4}}}{}$,
  computes an irreducible polynomial in $k[x]$ of degree $\ell$, at the
  expense of
  $\ell \times (\log q)^{5+o(q)} +
  \ell^{1+o(\ell)}\times (\log q)^{1+o(q)}$ elementary
  operations.
\end{lem}

\subsubsection{An extended algorithm}
\label{sec:an-extend-algor}

Let now $\ell$ be an odd Elkies degree, prime to $p\,(q-1)$. The integer
$\ell$ is thus odd. More specifically, it will become clear later on that
$\ell$ is the product of at most two prime powers. With the notations of
Section~\ref{sec:isog-large-degr}, we denote by $\sigma_q$ the image of the
Frobenius endomorphism $\phi _q$ in $\mathcal{O}_\kappa$. In this setting, let
$\mathfrak{l}$ be an ideal in $\mathcal{O}_\kappa$ above $\ell$ and containing
$\sigma_q - r$ where $r\in \Z/{\ell}\Z$ is a root of
$X^2 - \Tr(E)\,X + q \bmod \ell$.

Take for $A$ the point at infinity $O_{\tilde E}$ in the construction of
Section~\ref{sec:constr-citecl13}. We thus consider the polynomial
$f_{I,O_{\tilde E}}(x) = \psi(x)$, whose roots are the abscissas of points in
$\Ker I$.  Now, the factorizations $I = I_{\ell/m} \circ I_m$, where $I_m$
are isogenies of degree $m$ with $m$ any divisor of $\ell$, yield
$\Ker I_m\subset \Ker I$.  Consequently, the polynomial $\psi(x)$ splits as
\begin{displaymath}
  \psi(x) = \prod \limits _{m\,|\,\ell} \Psi_m(x)\,,
\end{displaymath}
with $\deg\Psi_m(x) = \varphi(m)/2$, the Euler's totient function of $m$.
Computing $I_m$ for $m\neq \ell$ by the same procedure as $I$, we can obtain
$\Psi_m(x)$. Dividing $\psi(x)$ by all of them, we are led to examine the
polynomial $\Psi_\ell(x)$, of degree $\varphi(\ell)/2$.
Here too, its irreducibility depends on the order of $r$ in the multiplicative
group $(\Z/{\ell}\Z)^\times$ because the length of the Galois orbit of $B$ is
determined by the relation $\phi_q(B) = r\,B$. The polynomial $\Psi_\ell(x)$
splits thus in factors of degree $d=\Ord_{\Z/{\ell}\Z}(r)/2$ or
$d=\Ord_{\Z/{\ell}\Z}(r)$ according to whether a power of $r$ is equal to $-1$
or not.  When $d=\varphi(\ell)/2$, the polynomial $\Psi_\ell(x)$ is therefore
irreducible.

For the same reasons, take for $B$ any non-zero point of ${E}(k)$ and take
$A=I(B)$, then the polynomial
\begin{displaymath}
  \nu_{I,A}(x) = f_{I, A}(x)\,/\,(\,x - x(B)\,)\,,
\end{displaymath}
of degree $\ell-1$, splits in factors $\Phi_m(x)$ of degree $\varphi(m)$ where
$m\neq 1$ divides $\ell$\,. In turn, $\Phi_\ell(x)$ splits in factors of
degree $\Ord_{\Z/{\ell}\Z}(r)$\,. The polynomial $\Phi_\ell(x)$ is thus
irreducible when $\Ord_{\Z/{\ell}\Z}(r)=\varphi(\ell)$.
\smallskip

The main part of this construction is thus to determine the isogenous curve
$\tilde{E}$ and the equations of the isogeny $I$ following
Theorem~\ref{thm:9}. Therefore, we can state this theorem.

\begin{thm}
\label{thm:10}
Given an ordinary elliptic curve $E$ defined over a finite field $k$ with
characteristic $p$ and cardinality $q=p^n$, and $\ell$, product of at most two
prime powers, an odd Elkies degree prime to $p\,(q-1)$ such that one of the
roots $r$ modulo $\ell$ of the Weil polynomial $X^2 - (\Tr E)\,X + q$ has
order $\varphi(\ell)$ in $(\mathbb{Z}/\ell \mathbb{Z})^\times$, there exists
an algorithm that computes two irreducible polynomials in $k$ of degree
$\varphi(\ell)$ and $\varphi(\ell)/2$ with time complexity
$O( n\, \ell + p^2+ q^{3/2})$ up to some polylogarithmic factors.\smallskip

With same complexity, this algorithm computes an irreducible polynomial of
degree $\varphi(\ell)/2$ when $\Ord_{\Z/{\ell}\Z}(r)=\varphi(\ell)/2$
and $-1\notin \langle r \rangle$\,.
\end{thm}

Note that, since $\ell$ is odd and $\Ord_{\Z/{\ell}\Z}(r)$ can not be larger
than the Carmichael function $\lambda(\ell)$, we have that
$\ell$ is either of the form $p_1^{e_1}$ or of the form
$p_1^{e_1}\,p_2^{e_2}$ where $p_1$ and $p_2$ are odd prime integers. The
former corresponds to the only possibility for $(\Z/{\ell}\Z)^\times$ to be
cyclic (\textit{e.g.}  $\lambda(\ell) = \varphi(\ell)$ and
$\Ord_{\Z/{\ell}\Z}(r)=\lambda(\ell) /2$), the latter (\textit{e.g.}
$\Ord_{\Z/{\ell}\Z}(r)=\lambda(\ell) = \varphi(\ell)/2$\,) follows from the
recursive definition of $\lambda$,
\begin{displaymath}
  \lambda (p_{1}^{e_{1}}\cdot p_{2}^{e_{2}}\cdots p_{k}^{e_{k}})
  =
  \Lcm \left(\lambda (p_{1}^{e_{1}}),\lambda (p_{2}^{e_{2}}),\ldots ,\lambda (p_{k}^{e_{k}})\right)\,.
\end{displaymath}
Also note that $\varphi(\ell)$ is nearly $\ell$ in Theorem~\ref{thm:10}, since the
necessary conditions on $\ell$ yields $\varphi(\ell) = p_1^{e_1-1}(p_1-1)$ or
$\varphi(\ell) = p_1^{e_1-1}(p_1-1)\,p_2^{e_2-1}(p_2-1)$.

\begin{remark}\label{rmk:gf2}
  Applied to the elliptic curve $E\,/\,\GF{2}: y^2+x\,y = x^3 + 1\,$, whose
  Weil polynomial is $X^2+X+2$, this method gives an infinite list of
  irreducible polynomials over $\GF{2}$, of degree
  \begin{multline*}
    3,\ 5,\ 6,\ 10,\ 11,\ 14,\ 21,\ 26,\ 28,\ 30,\ 33,\ 35,\ 39,\ 42,\ 52,\ 53,\ 54,\ 55,\ 56,\ 63,\ 66,\ \\
    70,\ 74,\ 75,\ 78,\ 81,\ 84,\ 89,\ 95,\ 96,\ 98, 105, 106, 108, 110, 112, 119, 131, 138\ldots
  \end{multline*}
  We give for the first ones the degree $\ell$ of the isogeny, the root $r$ of
  $X^2+X+2$ and its order modulo $\ell$ in the following table.\smallskip
  \begin{center}
    \begin{tabular}[rcl]{c||c|cc|c|c|c|c|cc|c|c|c|cc|c|c|c|c}
      $d$          &3&5&5&6&10&11&14&21&21&26&28&30&33&33&35&39&42&52\\\hline
      $\ell$       &7&11&11&7&11&23&29&43&43&53&29&77&67&67&71&79&43&53\\
      $r$          &3&6&4&3&6&13&21&24&18&14&21&59&55&11&31&66&18&14\\
      $\Ord_\ell(r)$&6&10&5&6&10&11&28&21&42&52&28&30&33&66&70&78&42&52
    \end{tabular}
  \end{center}
\end{remark}
\smallskip

\begin{remark}
  Similarly to Remark~\ref{rmk:gf2}, we can easily do an exhaustive search on
  degrees that are not reachable with this method, whatever the field or
  the curve are. The first ones are
  \begin{multline*}
    7,\ 13,\ 17,\ 19,\ 24,\ 25,\ 31,\ 32,\ 34,\ 37,\ 38,\ 43,\ 45,\ 47,\ 49,\
    57,\ 59,\ 61,\ 62,\ 64,\\ 67,\ 71,\ 73,\ 76,\ 77,\ 79,\ 85,\ 87,\ 91,\
    93,\ 94,\ 97,\ 101,\ 103,\ 104,\ 107,\ 109\ldots
  \end{multline*}
  For instance, degree $7$ is not possible because there is no integer $\ell$
  such that $\varphi(\ell)$ equals $7$ or $14$.
\end{remark}

\begin{remark}
  Theorem~\ref{thm:9} and Theorem~\ref{thm:10} can be extended to
  supersingular elliptic curves if we lift them together with a quadratic
  order (see for instance \cite{CH02}).
\end{remark}

\subsection{An example}
\label{sec:an-example}

We consider the finite field $\mathbb{F}_{16} = \mathbb{F}_2(v)$ such that
$v^4 +v + 1 =0$.
Let $E$ be the elliptic curve defined by
$E/\mathbb{F}_{16} \: : y^2 + xy = x^3 + v^6$.  Choose $\ell=73$, the Weil
polynomial of $E$ satisfies
\begin{displaymath}
X^2 + 3X + 16 \equiv (X-10)\,(X-60) \: \bmod \ell\,.
\end{displaymath}
The endomorphism ring of $E$ is isomorphic to the ring of integers
$\mathcal{O}$ of the quadratic field $\mathbb{Q}(\sqrt{-55})$. The class group
$\Cl(\mathcal{O})$ is cyclic of order 4. Let $\mathfrak{l}$ be the ideal of
$\mathcal{O}$ generated by $73$ and $\phi _{16} -60$. The set
$E[\mathfrak{l}]$ is a cyclic subgroup of $E$ of order $73$, closed under the
action of the Frobenius endomorphism. Let
$I \: : E \longrightarrow {E}/{E[\mathfrak{l}]}$ be the degree $73$ isogeny
with kernel $E[\mathfrak{l}]$. We give the first coordinate of $I$ as the
rational fraction ${\phi (x)}/{\psi (x)^2 }$.  Note that $\psi$ is a degree
$36$ irreducible polynomial since $60$ is a generator of the multiplicative
group $\mathbb{F}_{73}^{\times}$.

Let us compute $\psi(x)$. The ideal $\mathfrak{l}$ can be decomposed as
\begin{math}
  \mathfrak{p}\, \mathfrak{l} = \mathfrak{u}_2\,,
\end{math}
where $\mathfrak{p}= ( 2 , (\sqrt{-55} + 1)/2 )$ and
$\mathfrak{u}_2 = ( -\sqrt{-55}+23)/2$. We begin by lifting $E$ in the
$2$-adics such that $\End(E) =$ $\End(\mathcal{E})$ as
\begin{displaymath}
  \mathcal{E}: y^2 + xy = x^3 + 21\,v^3 + 261\,v^2 + 316\,v + 256 +
  O(2^{10})\,.
\end{displaymath}
In order to compute an equation of $\tilde{\mathcal{E}}$ and the differential
isogeny, we first construct the degree $2$ isogeny
\begin{math}
  {\mathcal{E}}\,/\,{\mathcal{E}[\mathfrak{p}]} \cong
  \tilde{\mathcal{E}} \longrightarrow \mathcal{E}\,,
\end{math}
we deduce
\begin{displaymath}
  {\mathcal{E}}\,/\,{\mathcal{E}[\mathfrak{p}]}:   y^2 = x^3 - (27 + O(2^{10}))x + 2(-224\,v^3 + 96\,v^2 - 160,v + 315) + O(2^{11}).
\end{displaymath}
In return, a $73$-isogenous curve $\tilde{E}$ to $E$ is given by
\begin{displaymath}
\tilde{E} \: : y^2 + xy = x^3 + v^{12}\,
\end{displaymath}
and the isogeny differential is
\begin{displaymath}
 c = 244v^3 + 164v^2 - 424v - 299 + O(2^{10})\,.
\end{displaymath}
Applying Algorithm~\ref{algo:DiffSolve} with
$U(t) = 4( 21v^3 + 261v^2 + 316v + 256 + O(2^{10}))\, t^4 + t +4$ and
$ V(t) = 4(v^3 + 123v^2 + 243v + 369 + O(2^{10})) )\,t^4 + t +4$ and $c$, and
reducing modulo $2$ we get the series $z(t)$. A final call to the
\textsc{half-gcd} algorithm yields the irreducible polynomial
\begin{multline*}
  \psi (x) =  \gf{F} + \gf{E}\, x +\gf{7}\, x^{2}+\gf{B}\, x^{3}+\gf{7}\, x^{4}+\gf{B}\, x^{5}+\gf{E}\, x^{6}+\gf{7}\, x^{7}+\gf{2}\, x^{9}+\gf{B}\, x^{10}+\gf{7}\, x^{13} \\
  +\gf{9}\, x^{14} +\gf{E}\, x^{15}+\gf{7}\, x^{16}+\gf{F}\, x^{17}+\gf{6}\, x^{18} +\gf{5}\, x^{19}+\gf{D}\, x^{20}+\gf{6}\, x^{21}+\gf{1}\, x^{22}+\gf{C}\, x^{23}+\gf{7}\, x^{24}\\
  +\gf{B}\, x^{26}+\gf{2}\, x^{27}+\gf{3}\, x^{28}+\gf{2}\, x^{29}+\gf{5}\, x^{30}+\gf{A}\, x^{31}+\gf{C}\, x^{32}+\gf{7}\, x^{33}+\gf{9}\, x^{34}+\gf{D}\, x^{35}+\, x^{36},
\end{multline*}
where for brevity's sake, we represent elements of $\mathbb{F}_{16}$ by
integers written in hexadecimal. In other words, we replace the element $v$ by the integer $2$, for instance $\gf{5}= v^2 +1$ and $\gf{C}=v^3+v^2$.

\appendix

\section{More on our differential equations}
\label{sec:diffeq}

In the previous sections, motivated by the explicit computation of
isogenies in characteristic $2$, we introduced and studied the following
nonlinear $2$-adic differential equation:
\begin{equation}
\label{eq:nonlinear1}
U \cdot z'{}^2 = V(z)
\end{equation}
where $U$ and $V$ are two series in $\Kt$ with $t$-adic valuation $1$.
Most of our attention was actually focused on the particular case where
the hypothesis \HU is satisfied, in which case Eq.~\eqref{eq:nonlinear1}
can be rewritten as follows:
\begin{equation}
\label{eq:nonlinear2}
t(t{-}4a) \cdot z'{}^2 = g^2 \cdot h(z).
\end{equation}
Here $a$ is a given element in $\ZZ^\times$ (or more generally
$\OK^\times$ where $K$ is a finite extension of $\QQ$), $g$ and $h$ are
given analytic functions and the unknown is~$z$.
In this appendix, we aim at revisiting our results and extracting
from them theoretical information about the structure of the solutions
of Eqs.~\eqref{eq:nonlinear1} and \eqref{eq:nonlinear2}.

\subsection{Some spaces of analytic functions}

As before, we fix a finite extension $K$ of $\QQ$ and denote by
$|\cdot|$ the norm on it, normalized by $|2| = 1/2$.
We set $\calV = \Kt$; it is the space of germs of analytic functions
around $0$. Given a positive real number $r$, we let $\calV_r$ be
the subset of $\calV$ defined by
$$\calV_r =
\Big\{\, \sum_{n=0}^\infty a_n t^n \quad \text{such that }
|a_n| \: r^n \text{ is bounded} \,\Big\}.$$
Series in $\calV_r$ converge when $|t| < r$ and thus define analytic
functions in the open disc of centre~$0$ and radius~$r$, denoted by
$B(r)$ in what follows.
Thanks to ultrametricity, these functions are moreover
all bounded on $B(r)$.
We equip $\calV_r$ with the Gauss norm $\Vert \cdot \Vert_r$ defined
by
$$\Vert f \Vert_r = \sup_{n \geq 0}\, |a_n| \: r^n
\quad \text{where }
f = \sum_{n=0}^\infty a_n t^n.$$
One can check that $\calV_r$ is complete with respect to $\Vert \cdot
\Vert_r$. Besides, it is obvious that, when $r \leq s$, we have
$\calV_s \subset \calV_r$ and $\Vert f \Vert_r \leq \Vert f
\Vert_s$ for all $f \in \calV_s$.
It is finally easy to check that the Gauss norm is compatible with
multiplication in the following sense: for all positive real number $r$
and all functions $f, g \in \calV_r$, we have $\Vert fg \Vert_r \leq
\Vert f\Vert_r \cdot \Vert g \Vert_r$.

\subsubsection*{The operator $\psi_+$}

In Section~\ref{subsec:linear}, we have introduced a linear automorphism
$\psi_+$ of $\Kt$ which takes a function $f \in \Kt$ to the unique solution
of the following linear differential equation:
$$t (t{-}4a) y' + (t{-}2a) y = f.$$
For all positive real number $r$, we set $\calV_{r,+} = \psi_+^{-1}
(\calV_r)$ and equip this space with the norm $\Vert \cdot \Vert_{r,+}$
defined by $\Vert f \Vert_{r,+} = \Vert \psi_+(f) \Vert_r$.
Clearly $\psi_+$ induces a bijective isometry $\psi_+ : \calV_{r,+} \to
\calV_r$. Besides, the equality
$t (t{-}4a)\:\psi_+(f)' + (t{-}2a)\:\psi_+(f) = f$
ensures that $\calV_{r,+} \subset \calV_r$ and
$$\textstyle \Vert f \Vert_r
\leq \max\big(\frac 1 2, r \big) \cdot \Vert \psi_+(f) \Vert_r
= \max\big(\frac 1 2, r \big) \cdot \Vert f \Vert_{r,+}$$
for all $r > 0$ and all $f \in \calV_{r,+}$.
The estimates of Proposition~\ref{prop:linear} allow us to derive
inequalities in the other direction.

\begin{prop}
\label{prop:controlnorm}
Let $r$ and $s$ be two real numbers such that $0 < r < s \leq 1$.
Then $\calV_s \subset \calV_{r,+}$ and, for all $f \in \calV_s$,
we have the estimation
$$\Vert f \Vert_{r,+} \leq \max\left(2, \,\frac 2 {\log(s/r)}\right)
\cdot \Vert f \Vert_s.$$
\end{prop}

\begin{proof}
We write $f = \sum\limits_{i=0}^\infty f_i t^i$ and $\psi_+(f) =
\sum\limits_{i=0}^\infty y_i t^i$. From Proposition~\ref{prop:linear},
we deduce that
$$|y_i| \leq 2 \cdot (i+1) \cdot \sup_{0 \leq k \leq i} |f_k|.$$
Multiplying by $r^i$ on each side and noticing that $|f_k| \:r^i
\leq \big(\frac r s\big)^i \: \Vert f \Vert_s$ for all $k \leq i$, we
derive  $|y_i| \: r^i \leq 2 {\cdot} (i+1) {\cdot} \big(\frac r s\big)^i
\: \Vert f \Vert_s$.
By calculus, we prove that, for any $a \in ]0,1[$, the maximum of the
function $x \mapsto
(x{+}1) \: a^x$ is reached for $x_0 = \max(0,\, -1 - 1/{\log a})$
and is equal to $1$ if $a \leq e^{-1}$ and to ${-1}/{(e a \log a)}$
otherwise. (Here $e \approx 2.718...$ denotes the natural base of
logarithms.) We deduce from
this that the function $x \mapsto (x{+}1) \: a^x$ is bounded from above
by $\max(1,\, {-1}/{\log a})$ on the interval $]0, +\infty[$.
The proposition follows, noticing that $\Vert f \Vert_{r,+} =
\Vert \psi_+(f) \Vert_r = \sup_{i \geq 0} |y_i| r^i$ by definition.
\end{proof}

\subsection{Generic radius of convergence}

We now come back to the nonlinear differential equations
\eqref{eq:nonlinear1} and \eqref{eq:nonlinear2}; we are interested in the
radius of convergence of their solutions. We recall that the radius of
convergence of a function $f \in \Kt$ is defined as the supremum of the
nonnegative real numbers $r$ for which $f \in \calV_r$. In the sequel, we
will denote it by $\RoC(f)$ for short. If $f = \sum_{n=0}^\infty a_n t^n$,
we have the classical explicit formula
$$\RoC(f) = \liminf_{n \to \infty}\,\, |a_n|^{-1/n}.$$
A general theorem indicates that the radii of convergence of the
solutions of Eq.~\eqref{eq:nonlinear1} are strictly positive as soon
as $U$ and $V$ have positive radii of convergence as well. The next
proposition makes this result effective in our setting.

\begin{prop}
\label{prop:RoC}
Let $U, V \in \Kt$ with $t$-adic valuation $1$.
We assume that $U \in \calV_r$ and $V \in \calV_s$ for some positive
real numbers $r$ and $s$. Let $z$ be the unique
solution of Eq.~\eqref{eq:nonlinear1} in $t \Kt$
(\emph{cf} Proposition~\ref{prop:solnonlinear}).
Then,
$$\RoC(z) \geq
\min\left(
\frac{r s^2 \cdot |U'(0)|^2}{\Vert U \Vert_r{\cdot}\Vert V \Vert_s},\,
\frac{r^2 \cdot |U'(0)|}{\Vert U \Vert_r}
\right).$$
\end{prop}

\begin{proof}
Performing the change of function $z(t) = y(\lambda t)$ for a well
chosen $\lambda$ in a suitable extension of $K$, we may assume without
loss of generality that $s = 1$. Up the rescaling $U$ and $V$ by the
same constant, we may further suppose that $\Vert V \Vert_1 = 1$.
We write
$$U = \sum\limits_{i=1}^\infty u_i t^i, \quad
V = \sum\limits_{i=1}^\infty v_i t^i, \quad
z = \sum\limits_{i=1}^\infty z_i t^i$$
with $u_i, v_i, z_i \in K$. Observe that $U'(0) = u_1$.
Moreover, by definition of the Gauss norm, we know that
$|u_i| \leq \Vert U \Vert_r \: r^{-i}$ and $|v_i| \leq 1$ for all $i$.
We set
$$\rho = \min\left(
\frac{r{\cdot}|u_1|^2}{\Vert U \Vert_r},\,
\frac{r^2{\cdot}|u_1|}{\Vert U \Vert_r}
\right)
\quad \text{and} \quad
C = \frac{\rho^2{\cdot}|v_1|}{|u_1|^2}.$$
We are going to prove by induction that
$|z_n| \leq C \cdot \rho^{-n}$ for all $n \geq 2$.
This will directly imply the proposition.

We consider an integer $n \geq 2$.
From Eq.~\eqref{eq:faadibruno} (obtained in the proof of
Proposition~\ref{prop:solnonlinear}),
we derive $|z_n| \leq |v_1|^{-1} \cdot \max(A, B)$ with
\begin{align*}
A & = \max_I \big(|z_1|^{k_1} \cdots |z_{n-1}|^{k_{n-1}}\big)  \\
B & = \Vert u \Vert_r \cdot
\max_J \big(|z_{j+1}| \cdot |z_{i-j+1}| \cdot r^{i-n}\big)
\end{align*}
where $I$ is the set of all tuples of nonnegative
integers $(k_1, \ldots, k_{n-1})$ such that $k_1 + 2 k_2 + \cdots +
(n{-}1) k_{n-1} = n$ and $J$ is the set of pairs $(i,j) \in \Z^2$
with $0 \leq j \leq i < n$ and $0 < j < n{-}1$ if $i = n{-}1$.
Let $(k_1, \ldots, k_{n-1}) \in I$.
From the induction hypothesis, we deduce
$$|z_1|^{k_1} \cdots |z_{n-1}|^{k_{n-1}}
\leq C_1^{k_1} \cdot C^{k_2 + \cdots + k_{n-1}} \cdot \rho^{-n}$$
where $C_1$ is defined by
$C_1 = \rho \cdot |z_1| = \rho \cdot \frac{|v_1|}{|u_1|} =
\sqrt{|v_1| \cdot C}$.
Our estimation then becomes
$$|z_1|^{k_1} \cdots |z_{n-1}|^{k_{n-1}}
\leq \left(\frac{C}{|v_1|}\right)^{\frac k 2 + k'}
\cdot |v_1|^{k + k'} \cdot \rho^{-n}$$
where, for simplicity, we have set $k = k_1$ and $k' = k_2 +
\cdots + k_{n-1}$.
On the other hand, from the definition of $\rho$,
we deduce that $\rho \leq \frac{r{\cdot}|u_1|^2}{\Vert U \Vert_r}
\leq |u_1|$; using then the definition of $C$, we find $C \leq |v_1|$.
Noticing further that $|v_1| \leq 1$ and that, necessarily,
$\frac k 2 + k' \geq 1$ and $k + k' \geq 2$, we end up with
$$|z_1|^{k_1} \cdots |z_{n-1}|^{k_{n-1}}
\leq \frac{C}{|v_1|} \cdot |v_1|^2 \cdot \rho^{-n}
= C \cdot |v_1| \cdot \rho^{-n}.$$
Taking the supremum over all $(k_1, \ldots, k_{n-1}) \in I$,
we are finally left with $A \leq C \cdot |v_1|\cdot \rho^{-n}$.

Let us now focus on $B$. We consider a pair $(i,j) \in J$. We first assume
that $i < n{-}1$. Clearly one of the indices $j{+}1$ or $i{-}j{+}1$ must be
strictly greater than $1$. We then deduce from the induction hypothesis that
$|z_{j+1}| \cdot |z_{i-j+1}| \cdot r^{i-n} \leq C_1 \cdot C \cdot \rho^{-i-2}
\cdot r^{i-n}$ where $C_1 = \rho \cdot \frac{|v_1|}{|u_1|}$ is the constant we
have introduced in the first part of the proof.  We rewrite the above
inequality as follows
$$\Vert U \Vert_r \cdot |z_{j+1}| \cdot |z_{i-j+1}| \cdot r^{i-n}
\leq \frac{\rho {\cdot} \Vert U\Vert_r}{r^2{\cdot}{u_1}}
\cdot \left(\frac{\rho} r\right)^{n-i-2} \cdot
C \cdot |v_1| \cdot \rho^{-n}.$$
From the definition of $\rho$, it is clear that
$\rho \leq \frac{r^2 \cdot |u_1|}{\Vert U \Vert_r}$, implying that
the first factor $\frac{\rho{\cdot}\Vert U\Vert_r}{r^2{\cdot}{u_1}}$
is at most $1$. Similarly, using $r{\cdot}|u_1| \leq \Vert U\Vert_r$,
we deduce that the quotient $\frac \rho r$ is at most $1$ as well.
Since the exponent $n{-}i{-}2$ is nonnegative by assumption, we find
\begin{equation}
\label{eq:majB}
\Vert U \Vert_r \cdot |z_{j+1}| \cdot |z_{i-j+1}| \cdot r^{i-n}
\leq C \cdot |v_1| \cdot \rho^{-n}
\end{equation}
in this case. We now consider the case where $i = n{-}1$.
By definition of $J$, we cannot have $j = 0$ or $j = n{-}1$. Thus
both indices $j{+}1$ and $i{-}j{+}1 = n{-}j$ are strictly greater
than $1$ and the induction hypothesis yields
\begin{align*}
\Vert U \Vert_r \cdot |z_{j+1}| \cdot |z_{n-j}| \cdot r^{-1}
 & \leq \Vert U \Vert_r \cdot C^2 \cdot \rho^{-n-1} \cdot r^{-1} \\
 & = \frac {C \cdot \Vert U \Vert_r}{\rho r \cdot |v_1|}
     \cdot C \cdot |v_1| \cdot \rho^{-n}
   = \frac {\rho \cdot \Vert U \Vert_r} {r \cdot |u_1|^2}
     \cdot C \cdot |v_1| \cdot \rho^{-n}
\end{align*}
the last equality coming from the very first definition of $C$.
It now follows from the definition of $\rho$ that the factor
$\frac {\rho \cdot \Vert U \Vert_r} {r \cdot |u_1|^2}$ is at most
$1$, implying that the inequality~\eqref{eq:majB} is also valid
when $i = n{-}1$. Taking the supremum over all $(i,j) \in J$, we
obtain $B \leq C \cdot |v_1|\cdot \rho^{-n}$.

Coming back to the estimation $|z_n| \leq |v_1|^{-1} \cdot \max(A, B)$,
we finally obtain $|z_n| \leq C \cdot \rho^{-n}$ and the induction
goes.
\end{proof}

\subsection{Overconvergence phenomena}

Under Assumptions \HU and \HV, Proposition~\ref{prop:RoC} shows that
the radius of convergence of the solution of Eq.~\eqref{eq:nonlinear1} is
at least $1/4$ (by taking $r = 1/4$ and $s = 1$).
Nonetheless, there do exist particular choices of
$U$ and $V$ for which the solution $z$ overconverges beyond this radius.
For example, when $U$ and $V$ are built from the equations of two isogenous
elliptic curves as in Eq.~\eqref{eq:isog2} (\emph{cf} page \pageref{eq:isog2}),
we know that $z$ has integral coefficients; hence, its radius of convergence
is at least $1$.
One may wonder if such examples are isolated or not; in what follows,
we prove a first result in this direction showing that the
overconvergence phenomenon we observed persists when the differential
equation is slightly perturbed.

From now on, we work with the differential equation \eqref{eq:nonlinear2}
(which is a particular case of Eq.~\eqref{eq:nonlinear1}).
We fix $h \in \calV_1$ with $t$-adic valuation $1$, \emph{i.e.} $h(0) = 0$
and $h'(0) \neq 0$. Let $\Omega$ denote the subset of $\Kt$ consisting of
series with non-vanishing constant coefficient. By
Proposition~\ref{prop:linear}, we know that Eq.~\eqref{eq:nonlinear2}
admits a unique solution $z_g \in t \Kt$ for all $g \in \Omega$.

\begin{prop}
\label{prop:isometrynonlinear}
Let $r \in ]0,1[$. We consider $g \in \calV_r$ satisfying the two
following assumptions:
\begin{enumerate}[(a)]
\item $g$ does not vanish on the open ball of centre $0$ and radius $r$
in an algebraic closure of $K$,
\item the solution $z_g$ of Eq.~\eqref{eq:nonlinear2} is in $\calV_r$.
\end{enumerate}
Then, for all $\gamma_1, \gamma_2 \in \calV_{r,+}$ such that
$$\Vert \gamma_i \Vert_{r,+} <
\min\left(\Vert z'_g \Vert_r, \,
\frac{\Vert g \Vert_r^2}{4{\cdot}\Vert z'_g \Vert_r}\right)
\quad \text{for } i \in \{1,\, 2\}$$
we have $\displaystyle
\frac{z_{g+\gamma_1} - z_{g + \gamma_2}}{t(t{-}4a)} \in \calV_r$ and
$\displaystyle
\left\Vert \frac{z_{g+\gamma_1} - z_{g + \gamma_2}}{t(t{-}4a)}
\right\Vert_r \leq \Vert \gamma_1 - \gamma_2 \Vert_{r,+}$.
\end{prop}

\begin{remark}
By Weierstrass Preparation Theorem, Assumption (a) is equivalent
to the fact that $\Vert g \Vert_r = |g(0)|$, \emph{i.e.} the maximum
of $g$ is reached at the origin. Besides, it implies that $g^{-1} \in
\calV_r$ as well and $\Vert g^{-1} \Vert_r = \Vert g \Vert_r^{-1}$.
\end{remark}

\begin{proof}[Proof of Proposition~\ref{prop:isometrynonlinear}]
We follow the proof of Proposition~\ref{prop:optimalprecision}. We
fix a positive integer $n$. We set $E_n = \calV_{r,+}/t^n \calV_{r,+}$
and $F_n = \calV_r/t^n\calV_r$ and equip them with the induced norms.
As $K$-vector spaces, both $E_n$ and $F_n$ are canonically isomorphic
to $\Kt/(t^n)$. However, the norms on them differ; we have
\begin{align*}
\Vert a_0 + a_1 t + \cdots + a_{n-1} t^{n-1} \Vert_{F_n}
& = \sup_{0 \leq i < n} |a_i|\: r^i \\
\Vert f \Vert_{E_n} & = \Vert \psi_{+,n}(f) \Vert_{F_n}
\end{align*}
for $a_0, \ldots, a_{n-1} \in K$ and $f \in E_n$.
As in Section~\ref{subsec:precision-analysis}, we consider the analytic
function
$$\begin{array}{rcl}
\theta_n : \quad W_n & \longrightarrow & F_n \smallskip \\
\gamma & \mapsto & \displaystyle \frac{z_{g+\gamma} - z_g}{t(t{-}4a)}
\end{array}$$
where the domain $W_n$ is the open subset of $E_n$ consisting of series
$\gamma$ for which $g{+}\gamma$ does not vanish at $0$.
Proposition~\ref{prop:dphin} shows that the differential of $\theta_n$
at a point $\gamma \in W_n$ is given by
\begin{equation}
\label{eq:dthetan}
d \theta_n(\gamma) : \deltag \mapsto
z'_{g+\gamma} \cdot (g{+}\gamma)^{-1} \cdot \psi_{+,n}(\deltag).
\end{equation}
Following \cite[Remark~2.6]{carova15}, we introduce a copy $\tilde E_n$
of $E_n$ equipped with the modified norm defined as follows:
$$\Vert f \Vert_{\tilde E_n} =
\frac{\Vert z'_g \Vert_{F_n}}{\Vert g \Vert_{F_n}} \cdot \Vert f \Vert_{E_n}.$$
Here $f$ denotes at the same time a series in $E_n$ and its copy
in $\tilde E_n$. In order to avoid similar confusions in the future,
we introduce the mapping $\Id : E_n \to \tilde E_n$ taking a series
in $E_n$ to its counterpart in $\tilde E_n$. We set $\tilde W_n =
\Id(W_n)$.
We deduce from Eq.~\eqref{eq:dthetan} that $\theta_n$ is solution of the
differential equation $d \theta_n = \tau_n \circ (\theta_n, \Id)$
where $\tau_n$ is defined by
$$\begin{array}{r@{\hspace{0.5ex}}c@{\hspace{0.5ex}}lcl}
\tau _n : \quad  F_n &\times& \tilde W_n
 &\longrightarrow&  \Hom({E_n},{F_n} ) \smallskip \\
(\zeta&, & \tilde \gamma) & \mapsto & \displaystyle
\left(\deltag \:\mapsto \:
\frac{z'_g + t(t{-}4a) \zeta' + 2(t{-}2a) \zeta}{g+\Id^{-1}(\tilde\gamma)}
\cdot \psi_{+,n}(\deltag)\right).
\end{array}$$
We consider a pair $(\zeta, \tilde \gamma) \in F_n \times \tilde E_n$
such that $\Vert \zeta \Vert_{F_n} < \Vert z'_g \Vert_{F_n}$ and
$\Vert \tilde \gamma \Vert_{\tilde E_n} < \Vert z'_g \Vert_{F_n}$.
Then,
$$\Vert z'_g + t(t{-}4a) \zeta' + 2(t{-}2a) \zeta \Vert_{F_n} =
\Vert z'_g \Vert_{F_n}.$$
Write $\gamma = \Id^{-1}(\tilde\gamma)$.
From the definition of the norm on $\tilde E_n$,
we derive $\Vert \gamma \Vert_{E_n} < \Vert g \Vert_{F_n}$, which
further implies that $\Vert \gamma \Vert_{F_n} < \Vert g \Vert_{F_n}$.
We deduce that $\Vert g{+}\gamma \Vert_{F_n} = | (g{+}\gamma)(0) | =
\Vert g \Vert_{F_n}$, showing then that
$\Vert(g{+}\gamma)^{-1}\Vert_{F_n} =
\Vert g{+}\gamma \Vert_{F_n}^{-1} = \Vert g \Vert_{F_n}^{-1}$. As a
consequence, we conclude that
$$\left\Vert
\frac{z'_g + t(t{-}4a) \zeta' + 2(t{-}2a) \zeta}{g+\Id^{-1}(\tilde\gamma)}
\right\Vert_{F_n} \leq \frac{\Vert z'_g \Vert_{F_n}}{\Vert g \Vert_{F_n}}$$
whenever $\Vert \zeta \Vert_{F_n} < \Vert z'_g \Vert_{F_n}$ and
$\Vert \tilde \gamma \Vert_{\tilde E_n} < \Vert z'_g \Vert_{F_n}$.
With the $\Lambda$-notation introduced in Eq.~\eqref{eq:defLambda},
we have proved that
$\Lambda(\tau_n)(x) \leq \log \Vert z'_g \Vert_{F_n} - \log \Vert g \Vert_{F_n}$
for all $x < \log \Vert z'_g \Vert_{F_n}$.
Applying \cite[Proposition~2.5]{carova15}, we deduce that
$$\forall x < \min \left(\log \Vert z'_g \Vert_{F_n}, \,
\log \frac{\Vert g \Vert_{F_n}} 2 \right), \quad
\Lambda(\theta_n)(x) \,\leq\, 2 x +
\log \left(\frac{\Vert g \Vert_{F_n}^2}{4{\cdot}\Vert z'_g \Vert_{F_n}}\right).$$
Applying now~\cite[Proposition~3.12]{carova14}, we find that
\begin{align}
& \theta_n\big(B_{E_n}(\delta)) = d\theta_n(0)\big(B_{E_n}(\delta)\big)
\subset B_{F_n}(\delta)
\label{eq:thetandelta} \\
\text{when}\quad &
\delta < \min \left(\Vert z'_g \Vert_{F_n}, \, \frac{\Vert g \Vert_{F_n}} 2,\,
\frac{\Vert g \Vert_{F_n}^2}{4{\cdot}\Vert z'_g \Vert_{F_n}}\right) =
\min \left(\Vert z'_g \Vert_{F_n},\,
\frac{\Vert g \Vert_{F_n}^2}{4{\cdot}\Vert z'_g \Vert_{F_n}}\right).
\nonumber
\end{align}
The last equality comes from the observation that
$\frac 1 2{\cdot}\Vert g \Vert_{F_n}$ is the geometrical mean between
the two others arguments in the minimum.
Passing to the limit on $n$ in Eq.~\eqref{eq:thetandelta}, we get the
proposition when $\gamma_1 = 0$. Finally, for a general $\gamma_1$, we
apply the same argument after having replaced $g$ by $g + \gamma_1$
and $\gamma_2$ by $\gamma_2 - \gamma_1$.
\end{proof}

\begin{cor}
\label{cor:semicontinuity}
Let $r \in ]0,1[$. We consider $g_0 \in \calV_r$ satisfying the two
following assumptions:
\begin{enumerate}[(a)]
\item $g$ does not vanish on the open ball of centre $0$ and radius $r$
in an algebraic closure of $K$,
\item the solution $z_{g_0}$ of Eq.~\eqref{eq:nonlinear2} is in $\calV_r$.
\end{enumerate}
Then, for all $\rho \in ]0,r[$ and all $g \in \calV_r$ such that
$$\Vert g - g_0 \Vert_r <
{\textstyle
\frac 1 2 \cdot \min\big(1, \log\big(\frac r \rho\big)\big)}
\cdot
\min\left(\Vert z'_{g_0} \Vert_\rho, \,
\frac{\Vert g_0 \Vert_\rho^2}{4{\cdot}\Vert z'_{g_0} \Vert_\rho}\right)$$
we have $z_g \in \calV_\rho$ and
$\left\Vert z_g - z_{g_0} \right\Vert_\rho
\leq \max\big(2, \,\frac 2 {\log(r / \rho)}\big)
\cdot \Vert g - g_0 \Vert_r$.
\end{cor}

\begin{proof}
We apply Proposition~\ref{prop:isometrynonlinear} with $r = \rho$,
$g = g_0$,
$\gamma_1 = g - g_0$ and $\gamma_2 = 0$ and then conclude by using
Proposition~\ref{prop:controlnorm} combined with the fact that
$\Vert t(t{-}4a) \Vert_\rho \leq 1$.
\end{proof}

Corollary~\ref{cor:semicontinuity} implies in particular that, for any
real number $r \in ]0,1[$, the function $\calV_r \to \R$ taking $g$ to
$\min(r, \RoC(z_g))$ is continuous (where the domain $\calV_r$ is equipped
with the topology of the norm $\Vert \cdot \Vert_r$).
By Proposition~\ref{prop:isometrynonlinear}, it is even locally
constant around each point $g$ such that $\RoC(z_g) < r$.
This theoretical result looks quite interesting to us and raises a new
range of questions. In particular, can we expect similar results for a
wider class of nonlinear $p$-adic differential equations?

\bibliographystyle{alphaabbr}

\newcommand{\etalchar}[1]{$^{#1}$}


\end{document}